\documentclass{article}
\usepackage[a4paper]{geometry}
\title{Random signed measures}
\author{Riccardo Passeggeri\footnote{Imperial College London, UK: riccardo.passeggeri@imperial.ac.uk}}

\RequirePackage{amsmath,amsthm,amsfonts,amssymb}
\RequirePackage{natbib}
\RequirePackage[colorlinks,citecolor=blue,urlcolor=blue]{hyperref}
\RequirePackage{graphicx}
\usepackage{bm}

\theoremstyle{plain}

\newtheorem{theorem}{Theorem}[section]
\newtheorem{lemma}[theorem]{Lemma}

\newtheorem{proposition}[theorem]{Proposition}
\newtheorem{corollary}[theorem]{Corollary}

\theoremstyle{remark}
\newtheorem{remark}[theorem]{Remark}

\newtheorem{definition}[theorem]{Definition}
\newtheorem*{example}{Example}

\begin{document}
	\maketitle
	
	\begin{abstract}
		Point processes and, more generally, random measures are ubiquitous in modern statistics. However, they can only take positive values, which is a severe limitation in many situations. In this work, we introduce and study random signed measures, also known as real-valued random measures, and apply them to constrcut various Bayesian non-parametric models. In particular, we provide an existence result for random signed measures, allowing us to obtain a canonical definition for them and solve a 70-year-old open problem. Further, we provide a representation of completely random signed measures (CRSMs), which extends the celebrated Kingman's representation for completely random measures (CRMs) to the real-valued case. We then introduce specific classes of random signed measures, including the Skellam point process, which plays the role of the Poisson point process in the real-valued case, and the Gaussian random measure. We use the theoretical results to develop two Bayesian nonparametric models -- one for topic modeling and the other for random graphs -- and to investigate mean function estimation in Bayesian nonparametric regression.
	\end{abstract}


	\section{Introduction}
A random measure is a measure-valued random variable. Point processes and, more generally, random measures play a crucial role in probability and statistical modelling. In Bayesian nonparametric statistics, random measures are often used as priors, or building blocks for priors (\textit{e.g.}~\citet{beraha2024,Bro1,Bro2,Catalano,griffin2017compound,James,Regazzini}). In survival analysis, random measures are employed to model the hazard function, which describes the instantaneous risk of an event occurring at a given time (\textit{e.g.}~\citet{DeBlasiPeccatiPrunster,LijoiNipoti,LoWeng1989}). In extreme values statistics, empirical processes, a subclass of random measures, play a crucial role as they capture the values over a threshold of the underlying process (\textit{e.g.}~\citet{BASRAK2009,CouVer22,kulik2020heavy,PassO}). These are just few examples of fields of applications of random measures. A complete account would be impossible due to their ubiquitous use in modern statistics.

Such a great variety of applications is due to the explicit, clear and flexible nature of the theoretical results of random measures, which allows them to model a vast array of different phenomena.  This led to the belief that such exceptionally nice results are specific to random measures and cannot be generalized, namely that it is impossible to obtain the same kind of results for more general objects, like for random signed measures. This belief is strongly supported by historical facts. The first rigorous development of the Poisson process goes back to the work of \citet{Levy}. More general point processes were considered by Palm in his dissertation \citet{Palm}, whose ideas were extended and made rigorous in the 50s by \citet{Khinchin} and \citet{Renyi56,Renyi57}. In the series of works \citet{PrekopaI,PrekopaII,PrekopaIII}, who was a student of R\'{e}nyi, for the first time tackled the problem of extending a collection of random variables indexed by sets to a real-valued random measure. He did it in a specific setting. Since then, a general necessary and sufficient result for the extension of a collection of random variables indexed by sets remained an open problem. Obtaining this result would mean providing a complete existence result for random signed measures. The importance of this extension result has also been extensively discussed in Daley-Vere Jones' book (see pages 18-22 in \citet{Daley}), among others.

In this paper, we solve this long-standing open problem by proving necessary and sufficient conditions for the extension of collections of random variables indexed by sets to random signed measures. This allows us to provide a canonical definition of random signed measures and this is the starting point of our work.

The paper is mainly divided in two parts. In the first one we develop the theory of random signed measures, while in the second one we present various examples and focus on Bayesian nonparametric applications.

In particular, in the first part of the paper we generalize the most important results for random measures, systematically extending the main results of the first three chapters of \citet{Kallenberg2}, including the mentioned extension result and the fact that the distribution of a random signed measure is determined by its finite-dimensional distributions. Further, we provide new types of results that are specific to random signed measures. For example, we show that a random signed measure possesses a Jordan decomposition in terms of two unique random measures.

We then restrict our focus to an important class of random signed measures: the completely random signed measures (CRSMs). Completely random measures (CRMs) constitute one of the most studied classes of random measures. CRMs are random measures whose increments are independent. \citet{Kingman} showed that a CRM has an almost surely unique representation in terms of three addends. The first one is a deterministic atomless measure. The second one is a fixed atomic component. The third one, called the ordinary component, is encoded in a Poisson point process on $S\times(0,\infty)$, where $(S,\mathcal{S})$ is a Borel space. This remarkable result provides a complete description of a CRM. CRMs are vastly used in Bayesian nonparametric analysis.

Similarly, here we define the CRSMs as random signed measures with independent increments. We generalize Kingman's representation and obtain that a CRSM has the same representation, where now the Poisson point process of the ordinary component is on $S\times\mathbb{R}$. This result provides a complete description of CRSMs.

Moreover, when there is no fixed component we show that the Jordan decomposition of a CRSM are two unique and \textit{independent} CRMs. This result allows to extend many results for CRMs to hold for CRSMs.

We remark that we use the general framework of \citet{Kallenberg2}, that we do not introduce any new assumption, and that random measures are examples of random signed measures. Thus, our results truly generalize the ones of random measures. Therefore, we debunk the common belief that the exceptionally nice theory of random measures cannot be generalized.

While proving these results, we also show using a counterexample that one of the main results of \citet{Daley} does not hold. We provide a correction of this result, which might also be of interest on its own.

In the second part of the paper, we focus on specific examples and applications. The first example is the Skellam point process. Poisson point processes are one of the most popular classes of random measures. One of the reasons for this popularity is that any simple point process with independent increments is a Poisson point process and vice-versa. We show an equivalent result for the signed case: any simple signed point process with independent increments is a Skellam point process and vice-versa. The Skellam point process is a signed point process obtained by the difference of two independent Poisson point processes. Hence, its finite-dimensional distributions are Skellam distributions. Skellam distributions were formally introduced in the 40s (\citet{Irwin,Skellam}) and they have been recently used for sports data, medicine, image analysis and differential privacy, see \textit{e.g.}~\citet{Aga,Karl3,Tomy}. In \citet{Bar} the authors introduced the Skellam process and used it in financial applications. Our definition of the Skellam (point) process includes all the definitions presented in the literature as special cases. 

The second main example is the Gaussian (random) measure, which arises naturally as limiting object in central limit theorems for measure-valued processes. Building on the work \citet{Horo-Gauss} we present it and discuss the conditions under which it is a random signed measure. 

For the third example we focus on CRSMs. We discuss how to (easily) construct CRSMs using any CRMs, thus providing a plethora of examples.

After these examples, we develop two Bayesian nonparametric models. The first one is a novel class of nonparametric prior distributions based on CRSMs. This class is a generalization and a real-valued extension of the model developed in \citet{Bro1}, further explored in \citet{Bro2,PassRM}. Under two different sets of assumptions, which accommodate different modeling needs, we provide an explicit formulation for the posterior, even in the case of multiple data points. Since we can now work with real values and not just positive ones, we can now use important distributions which were left out before, like the Gaussian distribution. Indeed, we present an explicit example which employs it.

The second model is a sparse random signed graph model. A signed graph is a graph in which each edge is associated with a positive or negative sign. Models for nodes interacting over such random signed graphs arise from various biological, social, and economic systems, see \citet{Siam}. In physics, signed graphs are a natural context for the Ising model (\textit{e.g.}~\citet{Ising}). Thanks to the work \citet{Derr}, there has been a recent surge of interest in signed graph neural networks. Our approach, which is similar to the one of \citet{Caron-and-Fox}, consists of viewing the random signed graph as a signed point process. After introducing the definition of exchangeable random signed measure, which is simply an extension of the one for random measures, we show that our random signed graph is exchangeable. Further, our model is based on a CRSM and we prove that the random signed graph is either sparse or dense depending on the L\'{e}vy measure of such CRSM.

In the last section of the paper we discuss the role of random signed measures in mean function estimation for Bayesian nonparametric regressions, which has been explored in \citet{DeBlasiPeccatiPrunster,GhosalVanDerVaart2007,IshwaranJames,LijoiNipoti,LoWeng1989,Naulet,PeccatiPrunster09,Wolpert}. We improve existing results and discuss how our results pave the way to study the problem in full generality, when the driving measure is a random signed measure.

We conclude this introduction by mentioning that, due to the lack of a general theoretical foundation, random signed measures have not been systematically studied in the literature and only scattered results have been obtained, see for example \citet{cohen,HELLMUND,Horo-Gauss,Jacob1995,Karr,Molchanov2019,nickl2009convergence}.

The paper is structured as follows. In Section \ref{Sec-measure} we develop the theory of random signed measures. In Section \ref{Sec-Examples} we discuss the examples, including the Skellam point process and the Gaussian random measure. Section \ref{Sec-applications} is devoted to the two Bayesian nonparametric models and the Bayesian mean function estimation problem. All the proofs and additional results are in the supplementary material file.

\section{Definition and properties of random signed measures}\label{Sec-measure}
We follow the general framework of \citet{Kallenberg2}. Let $(S,\mathcal{S},\hat{\mathcal{S}})$ be a localized Borel space, that is $(S,\mathcal{S})$ is a Borel space and $\hat{\mathcal{S}}$ is a localizing ring in $S$. Recall that a localizing ring $\hat{\mathcal{S}}$ is a ring contained in $\mathcal{S}$ such that, for every $B\in\hat{\mathcal{S}}$ and $C\in\mathcal{S}$ we have $B\cap C\in\hat{\mathcal{S}}$, and $\hat{\mathcal{S}}=\cup_n(\mathcal{S}\cap S_n)$ for some sets $S_n\uparrow S$ in $\hat{\mathcal{S}}$. Further, let $\mathcal{M}_S$ be the space of locally finite measures on $S$ and let $\mathcal{B}_{\mathcal{M}_S}$ be the $\sigma$-algebra in $\mathcal{M}_S$ generated by the projection maps $\pi_{B}:\mu\mapsto\mu(B)$, $B\in\mathcal{S}$. A random measure is a measurable map from  $(\Omega,\mathcal{F},\mathbb{P})$ to $(\mathcal{M}_{S},\mathcal{B}_{\mathcal{M}_S})$. 
We recall the definition of a countably additive set function on a ring.

\begin{definition}[Countably additive set function on a ring] \label{Def-signedmeasure-ring} A set function $\mu(A)$ defined on the elements of a ring $\mathcal{R}$ with values in $[-\infty,\infty]$ will be called countably additive, if $\mu(\emptyset)=0$ and if for every sequence $A_{1},A_{2}, . . .$ of disjoint sets of $\mathcal{R}$ for which $A=\bigcup_{k=1}^{\infty}A_{k}\in\mathcal{R}$ we have
	\begin{equation}\label{measure on a ring}
		\mu(A)=\sum_{k=1}^{\infty}\mu(A_{k}),
	\end{equation}
	and the relation (\ref{measure on a ring}) holds absolutely (namely independent of the order of its elements).
\end{definition}

Let $\mathsf{M}_{S}$ denote the space of all countably additive set functions on $\hat{\mathcal{S}}$ that are locally finite, namely they take values in $\mathbb{R}$. $\mathsf{M}_{S}$ is also called the space of signed Radon measures, though in the literature there is more than one definition of signed Radon measure. In this case we mean the dual space of $C_c(S)$, namely the space of continuous real-valued functions with compact support on $S$ (endowed with the natural locally convex topology). 

Let $\mathcal{B}_{\mathsf{M}_S}$ be the $\sigma$-algebra in $\mathsf{M}_S$ generated by the projection maps $\pi_{B}:\mu\mapsto\mu(B)$, $B\in\hat{\mathcal{S}}$, \textit{i.e.}~generated by the family of sets $\{\mu\in\mathsf{M}_S:\mu(B)<x\}$ where $B\in\hat{\mathcal{S}}$ and $x\in\mathbb{R}$. 

Let $\mathfrak{F}_{b}(\mathcal{S})$ be the set of bounded measurable functions from $\mathcal{S}$ to $\mathbb{R}$ with bounded support. In the following we show that $\mathcal{B}_{\mathsf{M}_S}$ is equivalently generated by all integration maps $\pi_f:\mu\mapsto\int fd\mu$ where $f\in\mathfrak{F}_{b}(\mathcal{S})$. As in \citet{Kallenberg2} we denote $\int fd\mu$ by $\mu(f)$. The first results we present discuss the nice properties of the space $(\mathsf{M}_{S},\mathcal{B}_{\mathsf{M}_S})$.

\begin{lemma}\label{lem-measurable-vector-space}
	$(\mathsf{M}_{S},\mathcal{B}_{\mathsf{M}_S})$ is a measurable vector space over the scalar field $\mathbb{R}$, that is $\mathsf{M}_{S}$ is a vector space and the operations of addition and scalar multiplication are measurable.
\end{lemma}

\begin{lemma}\label{lem-uncited-1}
	The collection $\{\pi_f\}_{f\in\mathfrak{F}_{b}(\mathcal{S})}$ generates $\mathcal{B}_{\mathsf{M}_S}$.
\end{lemma}

\begin{lemma}\label{lem-inclusion}
	We have $\mathcal{M}_{S}\subset\mathsf{M}_{S}$ and $\mathcal{B}_{\mathcal{M}_S}\subset \mathcal{B}_{\mathsf{M}_S}$. Moreover, the inclusion map $\mu\hookrightarrow\mu$ from $(\mathcal{M}_{S},\mathcal{B}_{\mathcal{M}_S})$ to $(\mathsf{M}_{S},\mathcal{B}_{\mathsf{M}_S})$ is measurable. Conversely, the maps $\mu\mapsto\mu_+$ and $\mu\mapsto\mu_+$ from $(\mathsf{M}_{S},\mathcal{B}_{\mathsf{M}_S})$ to $(\mathcal{M}_{S},\mathcal{B}_{\mathcal{M}_S})$ are measurable.
\end{lemma}

Lemma \ref{lem-inclusion} shows that the space of measure $(\mathcal{M}_{S},\mathcal{B}_{\mathcal{M}_S})$ is a subspace of $(\mathsf{M}_{S},\mathcal{B}_{\mathsf{M}_S})$. This shows that we are not restricting our analysis but on the contrary we are considering a more general space. We are now ready to define random signed measures.
\begin{definition}[Random signed measure]  A random signed measure on $S$ is a measurable map from $(\Omega,\mathcal{F},\mathbb{P})$ to $(\mathsf{M}_{S},\mathcal{B}_{\mathsf{M}_S})$. 
\end{definition}

We now present an auxiliary but fundamental lemma, which provides a useful characterization of the space $(S,\mathcal{S},\hat{\mathcal{S}})$.

\begin{lemma}\label{lem-countable-ring-existence}
	Let $(S,\mathcal{S},\hat{\mathcal{S}})$ be a localized Borel space, then there exists a countable ring of bounded Borel sets generating $\mathcal{S}$.
\end{lemma}
As shown in the appendix, an example of such generating countable ring is the collection of all the finite unions of the elements of a dissection system on $S$. 

Lemma \ref{lem-countable-ring-existence} is an important characterization of the space $(S,\mathcal{S},\hat{\mathcal{S}})$, because it allows for a more refined object as a countable generating set of $\mathcal{S}$. This in turn is useful in results like Theorem 2.15 in \citet{Kallenberg2}, which provides necessary and sufficient conditions for the extension of a collection of positive random variables indexed by sets \textit{over a ring} to a random measure. However, surprisingly, it has not been stated or used in \citet{Kallenberg2}.

The following theorem solves a 70-year-old open problem, which was first systematically attempted in \citet{PrekopaI,PrekopaII,PrekopaIII} and extensively discussed in \citet{Daley}. It provides necessary and sufficient conditions for the extension of a collection of real-valued random variables indexed by sets to a random signed measure, and it is a generalization of Theorem 2.15 in \citet{Kallenberg2}.

\begin{theorem}\label{thm-extension}
	Let $\mathcal{R}\subset\hat{\mathcal{S}}$ be any countable ring of bounded Borel sets generating $\mathcal{S}$. Given a collection of random variables $(\eta(U))_{U\in\mathcal{R}}$, there exists a random signed measure $\xi$ on $S$ with $\xi(U) \stackrel{a.s.}{=} \eta(U)$ for all $U\in\mathcal{R}$, iff
	\\ \textnormal{(i)}  $\eta(A\cup B)\stackrel{a.s.}{=}\eta(A)\cup\eta(B)$ for all $A,B\in \mathcal{R}$ disjoint,
	\\ \textnormal{(ii)} $\eta(A_n)\stackrel{a.s.}{\to}0$ as $A_n\downarrow\emptyset$ along $\mathcal{R}$,
	\\ \textnormal{(iii)} $\sup\big\{\sum_{k=1}^{n}|\eta(A_k)|:\, \textnormal{$n\in\mathbb{N}$, $A_k\in\mathcal{R}$ disjoint and\,$A_k\subset A$} \big\}<\infty$ a.s.~for all $A\in\mathcal{R}$.
	\\ In that case, $\xi$ is a.s.~unique.
\end{theorem}
Theorem \ref{thm-extension} represents a complete existence (and uniqueness) result for our random signed measures. Moreover, Theorem \ref{thm-extension} generalizes Theorem 2.15 in \citet{Kallenberg2} because conditions (i) and (ii) in both theorems are the same and condition (iii) in Theorem \ref{thm-extension} always holds for positive-valued random variables satisfying conditions (i) and (ii). To be precise condition (ii) in \citet{Kallenberg2} has the convergence in probability, but that is equivalent to a convergence almost surely as seen in Theorem 9.1.XV in \citet{Daley}.

In Chapter 13 of \citet{cohen}, which is based on \citet{Jacodbook1979} and \citet{JacShir03}, the authors considered a definition of random signed measure whose existence, as they pointed out, is not assured. Theorem \ref{thm-extension} establishes the existence of this random signed measure, thereby strengthening the foundation of their framework and paving the way for possible extensions.

\begin{corollary}\label{co-extension}
	Fix a generating ring $\mathcal{R}\subset\hat{\mathcal{S}}$. Then the probability measures $\mu_{B_1,...,B_n}$ on $\mathbb{R}^n$, $B_1, . . . , B_n \in\mathcal{R}$,  $n\in\mathbb{N}$, are finite-dimensional distributions of some random measure $\xi$ on $S$, if and only if they are consistent and satisfy
	\\ \textnormal{(i)}  $\mu_{A,B,A\cup B}((x,y,z):x+y=z)=1$ for all $A,B\in \mathcal{R}$ disjoint,
	\\ \textnormal{(ii)} $\mathbb{P}(\lim\limits_{n\to\infty}\xi(A_n)=0)=1$ as $A_n\downarrow\emptyset$ along $\mathcal{R}$,
	\\ \textnormal{(iii)} $\mathbb{P}(\sup\big\{\sum_{k=1}^{n}|\xi(A_k)|:\, \textnormal{$n\in\mathbb{N}$, $A_k\in\mathcal{R}$ disjoint and\,$A_k\subset A$} \big\}<\infty)=1$
\end{corollary}

A similar result to Corollary \ref{co-extension} is stated in Theorem 10 in \citet{Jacob1995}. In particular, the authors focus on a less general framework (where $S$ is a compact Hausdorff second countable topological space) and state the existence of a random measure $\xi$ under condition (i) and the condition that $\lim\limits_{n\to\infty}\mathbb{P}(\sup_{C\subset A_n}|\xi(C)|\leq t)=1$ as $A_n\downarrow\emptyset$ along $\mathcal{R}$ for every $t>0$. Then one can wonder if we can replace conditions (ii) and (iii) with this condition both in Corollary \ref{co-extension} and possibly also in Theorem \ref{thm-extension}. The answer is negative, because even in their framework that condition is not sufficient. In other words, Theorem 10 in \citet{Jacob1995} and consequently Corollary 11 in \citet{Jacob1995} do not hold. We prove this using a counterexample in the supplementary material.

The reason for working with countably additive set functions on $\hat{\mathcal{S}}$ is the following. If we consider signed measures on $\mathcal{S}$ taking values in $[-\infty,\infty]$ and define random signed measures as measurable maps taking values in the space of such signed measures then we have problems of the form $\infty-\infty$. For example, let $\mu_1$ and $\mu_2$ be two measures then $\mu_1(S)-\mu_2(S)$ might not be well defined. On the other hand, we can show that the difference of \textit{any} two random measures uniquely generates a random signed measure.

\begin{lemma}\label{lem-uncited-2}
	Given any two random measures $\eta$ and $\gamma$, there is an almost surely unique random signed measure $\xi$ such that $\xi(B)\stackrel{a.s.}{=} \eta(B)-\gamma(B)$ for every $B\in\hat{\mathcal{S}}$.
\end{lemma}

In the next result we show the converse. In particular, we provide a Jordan decomposition result for random signed measures.
\begin{lemma}\label{lem-Jordan}
	Let $\xi$ be a random signed measure. Then, there exists two unique random measures $\xi_+$ and $\xi_-$ such that for every $\omega\in\Omega$ and every $B\in\hat{\mathcal{S}}$
	\begin{align*}
		\xi_+(\omega,B)&=\sup_{E\in \mathcal{S}, E\subset B}\xi(\omega,E),\quad
		\xi_-(\omega,B)=\sup_{E\in \mathcal{S}, E\subset B}-\xi(\omega,E),
	\end{align*}
	\begin{equation*}
		and\quad\xi(\omega,B)=\xi_+(\omega,B)-\xi_-(\omega,B).
	\end{equation*}
\end{lemma}

We call $\xi_+$ and $\xi_-$ the\textit{ Jordan decomposition} of $\xi$. We let $|\xi|:=\xi_+ +\xi_-$ and call it the total variation of $\xi$. Note that $|\xi|$ is a random measure. Lemma \ref{lem-Jordan} has immediate implications.

\begin{corollary}\label{co-atom}
	Any random signed measure has at most countably many fixed atoms.
\end{corollary}

\begin{corollary}\label{co-uncited-3}
	For any random signed measure $\xi$, there exists a bounded measure $\nu$ on $S$, such that for every measurable function $f \geq 0$ on $S$,
	\begin{equation*}
		\nu f=0\Leftrightarrow|\xi|f\stackrel{a.s.}{=}0\Rightarrow \xi f\stackrel{a.s.}{=}0.
	\end{equation*}
\end{corollary}

In the following results we present a representation of locally finite countably additive set functions (resp.~random signed measures), which generalizes the representation for measures (resp.~random measures).
\begin{lemma}\label{lem-deterministic-representation}
	Any $\mu\in\mathsf{M}_S$ has the following decomposition
	\begin{equation}\label{eq-measure}
		\mu=\alpha_+-\alpha_-+\sum_{k\leq\kappa_+}\beta_{+,k}\delta_{\sigma_{+,k}}-\sum_{k\leq\kappa_-}\beta_{-,k}\delta_{\sigma_{-,k}}
	\end{equation}
	where $\alpha_+,\alpha_-\in\mathcal{M}_S$ and they are atomless, $\kappa_+,\kappa_-\in\mathbb{N}\cup\{0\}\cup\{\infty\}$, $\beta_{+,1},\beta_{-,1},\beta_{+,2},\beta_{-,2},...>0$, and distinct $\sigma_{+,1},\sigma_{-,1},\sigma_{+,2},\sigma_{-,2},...\in S$. Moreover, the Jordan decomposition of $\mu$ is given by the measures
	\begin{equation*}
		\mu_+=\alpha_+\sum_{k\leq\kappa_+}\beta_{+,k}\delta_{\sigma_{+,k}}\quad\text{and}\quad \mu_-=\alpha_-+\sum_{k\leq\kappa_-}\beta_{-,k}\delta_{\sigma_{-,k}}.
	\end{equation*}
	The representations for $\mu$, $\mu_+$, and $\mu_-$ are unique up to the order of terms, and we may choose $\alpha_+$, $\alpha_-$, $\kappa_+$, $\kappa_-$, and all $(\beta_{+,k},\sigma_{+,k})$ and  $(\beta_{-,k},\sigma_{-,k})$ to be measurable functions of $\mu$, $\mu_+$ and $\mu_-$. Conversely, any locally finite sum as in $(\ref{eq-measure})$, involving measurable functions $\alpha_+$, $\alpha_-$, $\kappa_+$, $\kappa_-$, $(\beta_{+,k},\sigma_{+,k})$, $(\beta_{-,k},\sigma_{-,k})$ as above, defined on some measurable space $(A,\mathcal{B})$, uniquely determine $\mu$, $\mu_+$, and $\mu_-$ as measurable functions on $A$.
\end{lemma}

\begin{corollary}\label{co-atomic decomposition}
	Let $\xi$ be a random signed measure on $S$, then
	\begin{equation}\label{representation}
		\xi=\alpha_+-\alpha_-+\sum_{k\leq\kappa_+}\beta_{+,k}\delta_{\sigma_{+,k}}-\sum_{k\leq\kappa_-}\beta_{-,k}\delta_{\sigma_{-,k}}
	\end{equation}
	where $\alpha_+$ and $\alpha_-$ are diffuse random measures, $\kappa_+$ and $\kappa_-$ are random variables in $\mathbb{N}\cup\{0\}\cup\{\infty\}$, $\beta_{+,1},\beta_{-,1},\beta_{+,2},\beta_{-,2},...$ are positive random variables, and $\sigma_{+,1},\sigma_{-,1},\sigma_{+,2},\sigma_{-,2},...$ are almost surely distinct random points in $S$.
\end{corollary}

We can further divide the atomic component of a random signed measure into two parts: a fixed atomic component and a non-fixed atomic component. As usual, we refer to a random signed measure as `atomless' if it has no fixed atomic component and `diffuse' if it lacks both fixed and non-fixed atomic components.

\begin{definition}
	[signed point process] Following the representation in equation $(\ref{representation})$, by a signed point process we mean a random signed measure such that $\beta_{+,k}$ and $\beta_{-,k}$ take value in $\mathbb{N}$ and $\alpha_+(S)=\alpha_-(S)=0$.
\end{definition}
A signed point process $\xi$ is said to be simple if all its weights $\beta_{+,k}$ and $\beta_{-,k}$ equal 1. Let $(T,\mathcal{T},\hat{\mathcal{T}})$ be a localized Borel space. By a marked signed point process on $S$ we mean a simple signed point process $\xi$ on a product space $S\times T$, such that $|\xi|(\{s\} \times T) \leq1$ for all $s\in S$. Note that the projection $\xi(\cdot\times T)$ is not required to be locally finite. In the next result we provide a fundamental property of random signed measures: the distribution of a random signed measure is determined by its finite-dimensional distributions.
\begin{proposition}\label{pro-fdd=d}
	Let $\mathcal{I}\subset\hat{\mathcal{S}}$ be a generating semi-ring. Then, for any random signed measures $\xi$ and $\eta$ on $S$, we have $\xi\stackrel{d}{=}\eta$ iff $(\xi(I_1),...,\xi(I_n))\stackrel{d}{=}(\eta(I_1),...,\eta(I_n))$, for every $I_1,...,I_n\in\mathcal{I}$ and $n\in\mathbb{N}$.
\end{proposition}

A random measure can be equivalently defined as a locally finite kernel from $\Omega$ to $S$ (see Lemma 1.14 in \citet{Kallenberg2}), namely as a function from $\xi:\Omega\times \mathcal{S}\mapsto[0,\infty]$ such that, for fixed $\omega\in\Omega$, $\xi(\omega,\cdot)$ is a locally finite measure on $S$ and, for fixed $B\in\mathcal{S}$, $\xi(\cdot, B)$ is a $\mathcal{F}$-measurable function. A similar definition applies to random signed measures.

We say that $\eta$ is a \textit{pre-kernel} from $\Omega$ to $S$ if $\eta$ is a map from $\Omega\times \hat{\mathcal{S}}\mapsto \mathbb{R}$ such that, for fixed $\omega\in\Omega$, $\eta(\omega,\cdot)$ is a countably additive set function on $\hat{\mathcal{S}}$ and, for fixed $B\in\hat{\mathcal{S}}$, $\eta(\cdot, B)$ is a $\mathcal{F}$-measurable function.

\begin{lemma}\label{lem-kernel}
	The following statements are equivalent:
	\\ \textnormal{(i)} $\xi$ is a pre-kernel from $\Omega$ to $S$,
	\\  \textnormal{(ii)} $\xi$ is a random signed measure on $S$,
	\\  \textnormal{(iii)} $\xi(\cdot,B)$ is $\mathcal{F}$-measurable for every $B\in\mathcal{I}$, where $\mathcal{I}\subseteq\hat{\mathcal{S}}$ is a semi-ring generating $\mathcal{S}$, and $\xi(\omega,\cdot)$ is a countably additive set function on $\hat{\mathcal{S}}$ for every $\omega\in\Omega$.
\end{lemma}

The possibility to define random signed measures using a kernel-like formulation, as it happens for random measures, shows that our definition is a natural generalization of the concept of random measure.

The next result, which follows from Theorem \ref{thm-extension}, provides the definition of a conditional random signed measure. For random measures, this result is used to prove the existence of the doubly stochastic and cluster processes, like the Cox process.
\begin{corollary}\label{co-uncited-4}
	For any random signed measure $\xi$ on $S$ and any sub $\sigma$-algebra $\mathcal{G}$ of $\mathcal{F}$, there exist $\mathcal{G}$-measurable kernels $\eta_+ = \mathbb{E}[\xi_+|\mathcal{G}]$ and $\eta_- = \mathbb{E}[\xi_-|\mathcal{G}]$ and pre-kernel $\eta = \mathbb{E}[\xi|\mathcal{G}]$ from $\Omega$ to $S$, such that $\eta_+(B) = \mathbb{E}[\xi_+(B)|\mathcal{G}]$ a.s. for all $B\in\hat{\mathcal{S}}$ and similarly for $\eta_-$ and $\eta$. If $\mathbb{E}[|\xi|]$ is locally finite, then $\eta_+$ and $\eta_-$ are random measures on $S$ and $\eta=\eta_+-\eta_-= \mathbb{E}[\xi|\mathcal{G}]$ is a random signed measure on $S$.
\end{corollary}

We now focus on random signed measures $\xi$ with the additional property that for any $B_1,...,B_n$, $n\in\mathbb{N}$, disjoint sets in $\hat{\mathcal{S}}$ the random variables $\xi(B_1),...,\xi(B_n)$ are jointly independent. We call this class of random measures \textit{completely random signed measures} (CRSMs) or equivalently \textit{random signed measures with independent increments}. We will now provide an almost surely unique representation for CRSMs, which extends the celebrated Kingman's representation in \citet{Kingman} for completely random measures (CRMs). One of the results we use to obtain such representation is a correction and extension of a result in \citet{Daley}. Indeed, the following example shows that Proposition 9.1.III (v) in \citet{Daley} is incorrect.
\begin{example}\label{counterexample}
	Consider the following purely atomic measure $\mu\in\mathcal{M}_S$. For simplicity we consider $S=\mathbb{R}$. In particular, let $\mu(\{s_0\})=a$ for some $s_0\in \mathbb{R}$ and some $a>0$, let  $\mu(\{s_n\})=\frac{1}{n^2}$ where $s_n=s_0+\frac{1}{n}$, for every $n\in\mathbb{N}$, and let $\mu$ be equal to 0 otherwise (namely $\mu(\mathbb{R}\setminus\bigcup_{n=0}^\infty\{s_n\})=0$). Consider the set $(a,b)$ where $b>a$. Let $\kappa_n=\mu(\{s_n\})$. It is possible to see that $\sum_{n=0}^\infty\delta_{(s_n,\kappa_n)}([s,s+1]\times(a,b))=0$. Now, consider a dissection system $(A_{nj})$ of measurable subsets of $[s,s+1]$ with the $A_{nj}$ containing $s$ being open. Then, we have that $\sum_{j}\delta_{\mu(A_{nj})}(a,b)=1$ for every $n\in\mathbb{N}$. Therefore, 
	\begin{equation*}
		\sum_{n=0}^\infty\delta_{(s_n,\kappa_n)}([s,s+1]\times(a,b))\neq\lim\limits_{n\to\infty} \sum_{j}\delta_{\mu(A_{nj})}(a,b),
	\end{equation*}
	which shows that equation $(9.1.9b)$ in Proposition 9.1.III (v) in \citet{Daley} does not hold. Many other similar counterexamples can be found.
\end{example}
The next result is a correction of Proposition 9.1.III (v) in \citet{Daley} and it might also be of interest on its own. For any connected interval $B\in  \mathcal{B} (\mathbb{R}\setminus\{0\})$ with extremes $a$ and $b$ and for any $0<\varepsilon<\frac{b-a}{2}$, we denote by $B_{\varepsilon}$ the open interval $(a+\varepsilon,b-\varepsilon)$ and by $B^{\varepsilon}$ the open interval $(a-\varepsilon,b+\varepsilon)$. Similarly we define $B_\varepsilon$ and $B^\varepsilon$ for any set $B\in  \mathcal{B} (\mathbb{R}\setminus\{0\})$.

\begin{lemma}\label{lem-correspondence-pos}
	There exists a one-to-one both ways measurable correspondence between a purely atomic measure $\mu\in\mathcal{M}_S$ and a counting measure on $S\times(0,\infty)$, $N_\mu$ say, given by $N_\mu(A\times B)=\sum_{i}\delta_{(s_i,\kappa_i)}(A\times B)$, where $\{(s_i,\kappa_i)_{i\in\mathbb{N}}\}$ are the atoms of $\mu$ with their respective values, thus $\kappa_i=\mu(\{s_i\})$, and satisfying
	\begin{equation}\label{D0}
		\int_{A\times(0,\infty)} sN_{\mu}(ds\times dx)<\infty,
	\end{equation}
	for every $A\in\hat{\mathcal{S}}$; the correspondence is the following: for every $A\in\hat{\mathcal{S}}$ and $B\in \mathcal{B} ((0,\infty))$ such that $0\notin\bar{B}$,
	\begin{equation}\label{D1}
		\mu(A)=	\int_{A\times(0,\infty)} sN_{\mu}(ds\times dx)
	\end{equation}
	and
	\begin{equation}\label{D2}
		N_{\mu}(A\times B)=\lim\limits_{\varepsilon\to 0}\lim\limits_{n\to\infty}\sum_{j}\delta_{\mu(A_{nj})}(B_{\varepsilon}),\quad\text{for $B$ open,}
	\end{equation}
	\begin{equation}\label{D3}
		N_{\mu}(A\times B)=\lim\limits_{\varepsilon\to 0}\lim\limits_{n\to\infty}\sum_{j}\delta_{\mu(A_{nj})}(B^{\varepsilon}),\quad\text{for $B$ closed,}
	\end{equation}
	where $(A_{nj})$ is a dissection system of measurable subsets of $A$.
\end{lemma}

In the next result we extend Lemma \ref{lem-correspondence-pos} to the random signed measure case.

\begin{lemma}\label{lem-correspondence}
	There exists a one-to-one almost sure correspondence between a purely atomic random signed measures $\xi$ and a marked signed point process on $\mathcal{S}$, $N_\xi$ say, satisfying
	\begin{equation}\label{eq-0}
		\int_{A\times \mathbb{R}\setminus\{0\}} |s|N_{\xi}(ds\times dx)<\infty\quad\textnormal{a.s.}
	\end{equation}
	for every $A\in\hat{\mathcal{S}}$; the correspondence is the following: almost surely we have that for every $A\in\hat{\mathcal{S}}$ and every $B\in \mathcal{B} (\mathbb{R}\setminus\{0\})$ such that $0\notin\bar{B}$
	\begin{equation}\label{eq-1}
		\xi(A)=	\int_{A\times \mathbb{R}\setminus\{0\}} sN_{\xi}(ds\times dx)
	\end{equation}
	and
	\begin{equation}\label{eq-2}
		N_{\xi}(A\times B)=\lim\limits_{\varepsilon\to 0}\lim\limits_{n\to\infty}\sum_{j}\delta_{\xi(A_{nj})}(B_{\varepsilon}),\quad\text{for $B$ open,}
	\end{equation}
	\begin{equation}
		N_{\xi}(A\times B)=\lim\limits_{\varepsilon\to 0}\lim\limits_{n\to\infty}\sum_{j}\delta_{\xi(A_{nj})}(B^{\varepsilon}),\quad\text{for $B$ closed,}
	\end{equation}
	where $(A_{nj})$ is a dissection system of measurable subsets of $A$. Moreover,  almost surely we have that for every $A\in\hat{\mathcal{S}}$ 
	\begin{equation*}
		\xi_+(A)=	\int_{A\times (0,\infty)} sN_{\xi}(ds\times dx),\quad\text{and}\quad 	\xi_-(A)=	-\int_{A\times (-\infty,0)} sN_{\xi}(ds\times dx)
	\end{equation*}
\end{lemma}

\begin{remark}
	Notice that for $\varepsilon$ large we have that $\lim\limits_{n\to\infty}\sum_{j}\delta_{\mu(A_{nj})}(B_{\varepsilon})$ might not exist. On the other hand, from the proof of Lemma \ref{lem-correspondence} we have that 
	\begin{equation*}
		\limsup\limits_{\varepsilon\to 0}\limsup\limits_{n\to\infty}\sum_{j}\delta_{\xi(\omega,A_{nj})}(B_{\varepsilon})=\liminf\limits_{\varepsilon\to 0}\liminf\limits_{n\to\infty}\sum_{j}\delta_{\xi(\omega,A_{nj})}(B_{\varepsilon}),
	\end{equation*}
	because for any $\varepsilon<\varepsilon^*$ and for every $n>n_\varepsilon$, for some $n_\varepsilon\in\mathbb{N}$, we have $\sum_{i}\delta_{(s_i,\kappa_i)}(A\times B)=\sum_{j}\delta_{\xi(\omega,A_{nj})}(B_{\varepsilon})$. Thus, one could equivalently work with $\limsup\limits_{\varepsilon\to 0}\limsup\limits_{n\to\infty}$ and $\liminf\limits_{\varepsilon\to 0}\liminf\limits_{n\to\infty}$ at the same time for defining $N_{\xi}$ and in the proof of Theorem \ref{thm-representation}.
\end{remark}

We are now ready to present the mentioned almost sure representation for CRSMs.
\begin{theorem}\label{thm-representation}
	A random signed measure $\xi$ on $S$ has independent increments iff a.s.
	\begin{equation}\label{eq-representation}
		\xi (B)=\sum_{j=1}^p X_j\delta_{s_j}(B)+\alpha(B)+\int_{\mathbb{R}\setminus\{0\}}x\Psi(B\times dx),\quad B\in\hat{\mathcal{S}}
	\end{equation}
	where $p\in\mathbb{N}\cup\{\infty\}$, $X_1,...,X_p$ are independent real-valued random variables, $s_1,...s_p$ are fixed points in $S$, $\alpha$ is given by the difference of two deterministic atomless measures in $\mathcal{M}_S$ and $\Psi$ is a Poisson process on $S\times (\mathbb{R}\setminus\{0\})$ independent of $X_1,...,X_p$. The intensity measure of $\Psi$, which we denote by $F$, satisfies 
	\begin{equation}\label{eq-intensity-poisson}
		\int_{\mathbb{R}\setminus\{0\}}1\wedge|x| F(B\times dx)<\infty,\quad\textnormal{for every $B\in\hat{\mathcal{S}}$.}
	\end{equation}
	and $F(\{s\}\times (\mathbb{R}\setminus\{0\}))=0$ for every $s\in S$. Moreover, the representation (\ref{eq-representation}) is almost surely unique.
\end{theorem}
\begin{remark}
	Representation (\ref{eq-representation}) can be rewritten as almost surely
	\begin{equation*}
		\xi (B)=\sum_{j=1}^p X_j\delta_{s_j}(B)+\alpha(B)+\sum_{(V_i,U_i)\in \Psi}U_i\delta_{V_i}(B),\quad B\in\hat{\mathcal{S}}
	\end{equation*}
	and the characteristic function of $\xi$ is given by
	\begin{equation*}
		\mathbb{E}\left[e^{it\xi(B)}\right]=\exp\bigg(it\alpha(B) +\int_{\mathbb{R}\setminus\{0\}}e^{itx} -1F(B\times dx)\bigg)\prod_{j=1}^{p}\mathbb{E}\left[e^{it X_j\delta_{s_j}(B)}\right].
	\end{equation*}
	where $t\in\mathbb{R}$ and $B\in\hat{\mathcal{S}}$. Since there is a one-to-one correspondence between the pair $(F,\alpha)$ and the distribution of $\xi$ up to a fixed component, we call $(F,\alpha)$ the \textit{characteristic pair} of $\xi$. The measure $F$ is usually called the \textit{intensity} or \textit{L\'{e}vy measure} of $\Psi$, or of $\xi$.
\end{remark}

We already discussed that our definition is the natural one because it is the extension of real-valued random set functions, see Theorem \ref{thm-extension}, and because it can be equivalently defined using a kernel-like formulation, see Lemma \ref{lem-kernel}. Now, Theorem \ref{thm-representation} is an additional confirmation that our definition of random signed measure is truly the most natural one. Indeed, even though a random signed measure is defined on $\hat{\mathcal{S}}$ we have shown that there exists a unique Poisson point process on the whole $S\times (\mathbb{R}\setminus\{0\})$ satisfying the natural property (\ref{eq-intensity-poisson}) for its intensity measure. Thus, Theorem \ref{thm-representation} is a proper generalization of Kingman's representation of CRMs.

\begin{corollary}\label{co-two-CRSM}
	For any atomless CRSM $\xi$ on $S$ we have that $\xi_+$ and $\xi_-$ are independent atomless CRMs. Moreover, almost surely
	\begin{align*}
		\xi_+(B)&=\alpha_+(B)+\int_{0}^{\infty}x\Psi(B\times dx),\quad B\in\hat{\mathcal{S}},\\
		\xi_-(B)&=\alpha_-(B)-\int_{-\infty}^{0}x\Psi(B\times dx),\quad B\in\hat{\mathcal{S}}.
	\end{align*}
\end{corollary}

\begin{corollary}\label{co-two-CRSM-2}
	Let $\eta_1$ and $\eta_2$ be two independent atomless CRMs with driving Poisson processes $\Psi_1$ and $\Psi_2$, respectively. Then, $\Psi_1$ and $\Psi_2$ are independent. 
	
	Moreover, let $\xi$ be given by $\xi(B)=\eta_1(B)-\eta_2(B)$, $B\in\hat{\mathcal{S}}$. Then, $\xi$ is an atomless CRSM with Jordan decomposition $\eta_1$ and $\eta_2$. In particular, if $F_1$ and $F_2$ are the L\'{e}vy measures of $\eta_1$ and $\eta_2$, respectively, then the L\'{e}vy measure $F$ of $\xi$ is given by $F_1$ on $S\times(0,\infty)$ and by $F_2$ on $S\times(-\infty,0)$, precisely for every $B\in\hat{\mathcal{S}}$ and $A\in\mathcal{B}(\mathbb{R}\setminus\{0\})$ we have $F(B\times A)=F_1(B\times (A\cap(0,\infty)))+F_2(B\times -(A\cap(-\infty,0)))$.
\end{corollary}

Corollaries \ref{co-two-CRSM} and \ref{co-two-CRSM-2} enable the extension of numerous and important results for CRMs to CRSMs. In particular, consider a CRSM $\xi$ and the related CRMs $\xi_+$ and $\xi_-$. In some cases, one can obtain a result for $\xi$ by applying the results for the CRMs $\xi_+$ and $\xi_-$ and then combine them using the independence of $\xi_+$ and $\xi_-$. 

The following result provides a representation for marked signed point processes and is an extension of Theorem 3.18 in \citet{Kallenberg2}. Recall that $(T,\mathcal{T},\hat{\mathcal{T}})$ is a localized Borel space. We say that a random signed measure $\xi$ on the product space $S\times T$ has independent $S$-increments, if for any disjoint sets $B_1, . . . , B_n\in\hat{\mathcal{S}}$, $n\in\mathbb{N}$, the random measures $\xi(B_1 \times\cdot),...,\xi(B_n\times\cdot)$ on $T$ are independent. Let $T^{\Delta}=T\cup\{\Delta\}$, where the point $\Delta\notin T$ is arbitrary. For $\mathcal{M}_T\setminus\{0\}$ we mean the space of measures $\mathcal{M}_T$ without the zero measure.
\begin{proposition}\label{pro-marked-signed}
	A $T$-marked signed point process $\xi$ on $S$ has independent $S$-increments iff a.s.
	\begin{equation*}
		\xi=\eta_+-\eta_-+\sum_{k}(\delta_{s_k}\otimes\delta_{\tau_k})
	\end{equation*}
	for some independent Poisson processes $\eta_+$ and $\eta_-$ on $S\times T$ with $E[\eta_+(\{s\}\times T)]\equiv 0$ and $E[\eta_-(\{s\}\times T)]\equiv 0$, some distinct points $s_1,s_2,...\in S$, and some independent random elements $\tau_1,\tau_2,...\in T^{\Delta}$ with distributions $\nu_1,\nu_2,...$, such that the measures
	\begin{equation*}
		\rho_1=\mathbb{E}[\eta_+]+\sum_{k}(\delta_{s_k}\otimes\nu_{k}),\quad\text{and}\quad 	\rho_2=\mathbb{E}[\eta_-]
	\end{equation*}
	on $S\times T$ are locally finite. Here $\rho_1$ and $\rho_2$ determine uniquely $\mathcal{L}(\xi)$, and any measures $\rho_1$ and $\rho_2$ on $\mathcal{M}_{S\times T}$ with $\sup_{s}\rho_1(\{s\}\times T)\leq 1$ and  $\sup_{s}\rho_2(\{s\}\times T)=0$ may occur. Further, $\eta_+$ and $\eta_-$ are the Jordan decomposition of the non-fixed atomic component of $\xi$.
\end{proposition}

\begin{theorem}\label{thm-ultra-representation}
	A random signed measure $\xi$ on $S\times T$ has independent $S$-increments iff a.s.
	\begin{equation}\label{representation-ultra}
		\xi=\alpha+\sum_{k}(\delta_{s_k}\otimes\beta_k)+\int\int (\delta_s\otimes\mu)\eta_+(ds\,d\mu)-\int\int (\delta_s\otimes\mu)\eta_-(ds\,d\mu)
	\end{equation}
	for some $\alpha\in\mathsf{M}_{S\times T}$ with $\alpha_+(\{s\}\times T)=\alpha_-(\{s\}\times T)=0$, some $\eta_+$ and $\eta_-$ independent Poisson processes on $S\times(\mathcal{M}_T\setminus\{0\})$ with $\mathbb{E}[\eta_+(\{s\}\times (\mathcal{M}_T\setminus\{0\}))]=\mathbb{E}[\eta_-(\{s\}\times (\mathcal{M}_T\setminus\{0\}))]=0$, some distinct points $s_1,s_2,...\in S$, and some independent random signed measures $\beta_1,\beta_2$ on $\hat{\mathcal{T}}$ with distributions $\nu_1, \nu_2,... $, such that
	\begin{equation*}
		\int\mu(C)\wedge 1 \rho(B\times d\mu)<\infty,\quad B\in\hat{\mathcal{S}},\, C\in\hat{\mathcal{T}},
	\end{equation*}
	where $\rho=\rho_1+\rho_2$ and $\rho_1=\mathbb{E}[\eta_+]+\sum_{k}(\delta_{s_k}\otimes\nu_{k})$ and $\rho_2=\mathbb{E}[\eta_-]$. The triplet $(\alpha,\rho_1,\rho_2)$ is unique and determines $\mathcal{L}(\xi)$, and any $\alpha$, $\rho_1$ and $\rho_2$ with the stated properties may occur. Further, $\eta_+$ and $\eta_-$ are the Jordan decomposition of the non-fixed atomic component of $\xi$.
\end{theorem}

\begin{lemma}\label{lem-Borel-disjoint-union}
	Let $-\mathcal{M}_T=\{-\mu:\mu\in\mathcal{M}_T\}$ and let 
	\begin{equation*}
		\mathfrak{B}_{ \mathcal{M}_T\cup-\mathcal{M}_T\setminus\{0\} }=\{A_1\cup A_2:A_1\in \mathcal{B}_{\mathcal{M}_T\setminus\{0\}},A_2 \in\mathcal{B}_{-\mathcal{M}_T\setminus\{0\}}\}.
	\end{equation*}
	Then, $( \mathcal{M}_T\cup-\mathcal{M}_T\setminus\{0\},\mathfrak{B}_{ \mathcal{M}_T\cup-\mathcal{M}_T\setminus\{0\} },\hat{\mathfrak{B}}_{ \mathcal{M}_T\cup-\mathcal{M}_T\setminus\{0\} } )$ is a localized Borel space with $\hat{\mathfrak{B}}_{ \mathcal{M}_T\cup-\mathcal{M}_T\setminus\{0\} }=\{A_1\cup A_2:A_1\in \hat{\mathcal{B}}_{\mathcal{M}_T\setminus\{0\}},A_2 \in\hat{\mathcal{B}}_{-\mathcal{M}_T\setminus\{0\}}\}$.
\end{lemma}

\begin{corollary}\label{co-incited-5}
	The representation $(\ref{representation-ultra})$ can be written as
	\begin{equation*}
		\xi=\alpha+\sum_{k}(\delta_{s_k}\otimes\beta_k)+\int\int (\delta_s\otimes\mu)\eta(ds\,d\mu)
	\end{equation*}
	where $\eta$ is a Poisson process on $S\times(\mathcal{M}_T\cup-\mathcal{M}_T\setminus\{0\})$. Let $F$ the L\'{e}vy measure of $\eta$. Then, the characteristic function of $\xi$ is
	\begin{align*}
		&\mathbb{E}\left[e^{it\xi(B\times C)}\right]\\&=\exp\bigg(it\alpha(B\times C) +\int_{\mathcal{M}_T\cup-\mathcal{M}_T\setminus\{0\}}e^{it\mu(C)} -1F(B\times d\mu)\bigg)\prod_{i=1}^{k}\mathbb{E}\left[e^{it \beta_i(C)\delta_{s_i}(B)}\right],
	\end{align*}
	where $t\in\mathbb{R}$, $B\in\hat{\mathcal{S}}$ and $C\in\hat{\mathcal{T}}$. When $\xi$ has no fixed atoms, then $\xi_+$ and $\xi_-$ are independent and they are given by
	\begin{equation*}
		\xi_+=\alpha_++\int\int (\delta_s\otimes\mu)\eta_+(ds\,d\mu),\quad\text{and}\quad 	
		\xi_-=\alpha_-+\int\int (\delta_s\otimes\mu)\eta_-(ds\,d\mu).
	\end{equation*}
\end{corollary}

\begin{corollary}\label{co-ultra-two-CRSM-2}
	Let $\gamma_1$ and $\gamma_2$ be two independent atomless random measures on $S\times T$ with independent $S$-increments and with driving Poisson processes $\Psi_1$ and $\Psi_2$, respectively. Then, $\Psi_1$ and $\Psi_2$ are independent. 
	
	Moreover, let $\xi$ be given by $\xi(B)=\eta_1(B)-\eta_2(B)$, $B\in\hat{\mathcal{S}}\otimes\hat{\mathcal{T}}$. Then, $\xi$ is an atomless random signed measure on $S\times T$ with independent $S$-increments with Jordan decomposition $\gamma_1$ and $\gamma_2$. In particular, if $F_1$ and $F_2$ are the L\'{e}vy measures of $\gamma_1$ and $\gamma_2$, respectively, then the L\'{e}vy measure $F$ of $\xi$ is given by $F_1$ on $S\times(\mathcal{M}_T\setminus\{0\})$ and by $F_2$ on $S\times(-\mathcal{M}_T\setminus\{0\})$, precisely for every $B\in\hat{\mathcal{S}}$ and $A\in\mathfrak{B}_{\mathcal{M}_T\cup-\mathcal{M}_T\setminus\{0\}}$ we have $F(B\times A)=F_1(B\times (A\cap\mathcal{M}_T))+F_2(B\times -(A\cap-\mathcal{M}_T))$.
\end{corollary}

\begin{remark}
	In this paper we use the convention that the L\'{e}vy measure is a measure on $S\times\mathbb{R}\setminus\{0\}$ in Theorem \ref{thm-representation} and on $S\times(\mathcal{M}_T\cup-\mathcal{M}_T\setminus\{0\})$ in Theorem \ref{thm-ultra-representation}. However, it is equivalent to work with the L\'{e}vy measure as a measure on $S\times\mathbb{R}$ and on $S\times(\mathcal{M}_T\cup-\mathcal{M}_T)$ such that it takes value $0$ at $0$, that is $F(S\times\{0\})=0$ (here $0$ stands for both $0\in\mathbb{R}$ and the zero measure $0\in\mathcal{M}_T$).
\end{remark}

\section{Examples}\label{Sec-Examples}
\subsection{A primary example: the Skellam point process}
The Poisson point process is the primary example of a random measure. This is because any simple random measure with independent increments is a Poisson point process and any Poisson point process is a  simple random measure with independent increments (see Theorem 3.17 in \citet{Kallenberg2}). For random signed measures this role is played by the Skellam point process, as shown in Proposition \ref{pro-Skellam} below.

The Skellam distribution was formally introduced in \citet{Irwin} and in \citet{Skellam}. Recently, it has been used for sports data, medicine, image analysis and particularly in differential privacy, where \citet{Aga} introduced the Skellam mechanism for differential learning, to name a few. In \citet{Bar} the authors introduced the Skellam process and used it in financial applications (see also \citet{Koopman}). We refer to \citet{Karl3} and \citet{Tomy} for two recent reviews. Here we provide a general definition of the Skellam (point) process, which includes all the definitions presented in the literature as special cases.
\begin{definition}
	[Skellam point process] The Skellam point process is a signed point process whose Jordan decomposition are two independent Poisson point processes.
\end{definition}
\noindent Thus, the Skellam point process $\xi$ is obtained by the difference of two independent Poisson point processes and $\xi(B)$ is distributed according to the Skellam distribution, for every $B\in\hat{\mathcal{S}}$.  There is no restriction on the two independent Poisson point processes; in particular, they can have any intensities. We stress that such a general definition is possible thanks to the results of the previous section. 

We are now ready to present the result on the equivalence between the Skellam point process and the simple signed point process with independent increments.

\begin{proposition}\label{pro-Skellam}
	Let $\xi$ be a simple signed point process with independent increments. Then, $\xi$ is a Skellam point process. Conversely, let $\eta$ be a Skellam point process. Then, $\eta$ is a simple signed point process with independent increments.
\end{proposition}

\subsection{Gaussian random measure}
An important example of random signed measure is the Gaussian random measure. Here we define the Gaussian random measure (Grm) as a Gaussian process $(\xi(A))_{A\in\mathcal{S}}$ such that, for each $\omega\in\Omega$, $\xi(\omega,\cdot)$ is a finite signed measure on $\mathcal{S}$. Thus, $\xi$ is a random signed measure according to our definition. Such Gaussian random measures are the limiting object of various central limit theorems for measure-valued processes (\citet{Karr}). The properties of such random signed measures have been studied in \citet{Horo-Gauss}, among others. Here we summarize his two main results that characterize the Grm. Assume that the Grm has mean zero. The covariance kernel $\nu_0$ of a Grm $\xi$ is defined on $\mathcal{S}\times\mathcal{S}$ by $\nu_0(A,B)=\mathbb{E}[\xi(A)\xi(B)]$, $A,B\in\mathcal{S}$.

\begin{theorem}[from Theorem 1.2 in \citet{Horo-Gauss}]\label{thm-Horo-1} A symmetric, finite, signed bimeasure $\nu_0$ on $\mathcal{S}\times \mathcal{S}$ is the covariance kernel of a Grm if and only if the following two conditions hold:
	\\ \textnormal{(i)} $\nu_0$ is positive definite
	\\ \textnormal{(ii)}  $\exists$  $k>0$ s.t.~$\sum_{i}\nu_0^{1/2}(B_i,B_i)\leq k $ for any countable partition of $S$ by sets $B_i\in\mathcal{S}$.
\end{theorem}
\begin{theorem}[from Theorem 1.3 in \citet{Horo-Gauss}]\label{thm-Horo-2} Let $\xi$ be a Grm. There exists a Gaussian process $(W(x))_{x\in S}$ having mean 0 and constant variance, such that, a.s., $\xi(A) = \int_{A} W(x)\mu(dx)$, $A\in\mathcal{S}$, where $\mu$ is the finite measure $\mu(A)=\mathbb{E}[|\xi(A)|]$.
\end{theorem}

By definition if $(A_n)_{n\in\mathbb{N}}$ is a pairwise disjoint sequence in $\mathcal{S}$ we have
\begin{equation}\label{Grm}
	\xi(\omega,\cup_{n=1}^\infty A_n)=\sum_{n=1}^{\infty}\xi(\omega,A_n)
\end{equation}
for all $\omega\in\Omega$. We point out that there is another definition of Gaussian random measure, which following \citet{Horo-Gauss} we call Gaussian quadratic random measure (Gqrm), which consists of a Gaussian process $(\eta(A))_{A\in\mathcal{S}}$ for which $(\ref{Grm})$ holds in quadratic mean. Thus, a Grm is always a Gqrm, but a Gqrm does not need to be a Grm, \textit{e.g.}~the Brownian stochastic integral $\int_{A} dW(t)$ is a Gqrm but not a Grm. 

\begin{remark}
	We can define the Gaussian measure as a Gaussian process $(\xi(A))_{A\in\hat{\mathcal{S}}}$ such that for each $\omega\in\Omega$ it is a countably additive real-valued set function on $\mathcal{S}$. Such definition is more general than the one of \citet{Horo-Gauss} and it is still a random signed measure according to our definition. Then, it might be possible to obtain results like Theorems \ref{thm-Horo-1} and \ref{thm-Horo-2} for this more general definition of Grm. We leave this to further research.
\end{remark}

\subsection{Many examples of CRSMs}
Given that a CRSM is the difference of two independent CRMs we can easily obtain numerous examples of CRSMs and of intensity measures of CRSMs. For example, if $\eta_1$ and $\eta_2$ are two independent CRMs, then the intensity measure $F$ of a CRSM $\xi=\eta_1-\eta_2$ is given by the intensity measures $F_1$ and $F_2$ as shown in Corollary \ref{co-two-CRSM}. In particular, let $\tilde{F}_2$ be a measure on $S\times(-\infty,0)$ s.t.~$\tilde{F}_2(B\times -A)=F(B\times A)$ where $B\in\hat{\mathcal{S}}$ and $A\in\mathcal{B}((0,\infty))$. Then, for every $F(B\times \cdot)$-integrable function $g$ we have
\begin{equation*}
	\int_{\mathbb{R}\setminus\{0\}}g(x)F(B\times dx)= 	\int_{0}^\infty g(x)F_1(B\times dx)+	\int_{-\infty}^0g(x)\tilde{F}_2(B\times dx).
\end{equation*}

\subsection{Random (signed) measures vs independently scattered random measures}
An independently scattered random measure (isrm) $\Lambda=(\Lambda(A))_{A\in\hat{\mathcal{S}}}$ is a collection of real random variables such that $\Lambda(A)$ is infinitely divisible for every $A\in\hat{\mathcal{S}}$, $\Lambda(A_1),...,\Lambda(A_n)$ are independent when $A_1,...,A_n\in\hat{\mathcal{S}}$ are disjoint, and $\Lambda(\cup_{n=1}^\infty A_n)\stackrel{a.s.}{=}\sum_{n=1}^\infty \Lambda(A_n)$ when $A_1,A_2,...$ are disjoint and $\cup_{n=1}^\infty A_n\in\hat{\mathcal{S}}$. Recall that a random signed measure $\xi$ satisfies the property that $\xi(\cup_{n=1}^\infty A_n)=\sum_{n=1}^\infty \Lambda(A_n)$ for all $A_1,A_2,...$ disjoint with $\cup_{n=1}^\infty A_n\in\hat{\mathcal{S}}$ a.s.. 

The first two properties of the isrm are very restrictive compared to the one of random signed measure, while the last property of isrm is similar to the one of random signed measure. In particular, the difference between the last property of isrm and the one of random signed measures is subtle. For isrm, each collection $A_1,A_2,...$ has a zero measure set for which that equality does not hold. The union of these zero measure sets might not be a zero measure set, because there are unaccountably many collections of sets in $\hat{\mathcal{S}}$. On the other hand, for random (signed) measures there is only one zero measure set for which the equality does not hold for all collections of sets in $\hat{\mathcal{S}}$.

While subtle, this difference holds profound significance. It enables us to define a random (signed) measure as a singular stochastic entity, rather than a collection of random variables. This pivotal shift opens the door to discussions regarding the measurability of random (signed) measures and of their representations. It allows us to delve into topics such as the convergence of random measures (as explored in Chapter 4 of \citet{Kallenberg2}), and it allows the derivation of results in the same complete form as those presented in this paper. All things that are not possible for isrm. 

It is worth noting that this difference is the same difference between conditional probabilities and regular conditional probabilities. Moreover, note that Theorem \ref{thm-extension} tells us that an isrm can be extended to a random signed measure if and only if it satisfies conditions (i), (ii) and (iii) there.  In particular, the extension of an isrm is a CRSM.

On the other hand, one of the main advantages of isrm is the capability to incorporate a Gaussian component and a greater number of small jumps, all while maintaining independent increments.
\section{Applications}\label{Sec-applications}
\subsection{Sentiment topic modelling}\label{Sec-Bay}
In this section we present a novel class of nonparametric prior distributions based on CRSMs, which extends the one developed in \citet{Bro1}, further explored in \citet{Bro2} and \citet{PassRM} among others. In particular, let the prior be modeled as $\Theta:=\sum_{k=1}^{K}\theta_{k}\delta_{\psi_{k}}$, where $K$ may be either finite or infinite and where $(\theta_{k},\psi_{k})$ is a pair consisting of the frequency (or rate) of the $k$-th trait together with its trait $\psi_{k}$, which belongs to some complete separable metric space $\Psi$ of traits. Notice that $\psi_{1},...,\psi_{K}$ include both the fixed and non-fixed atoms. This representation follows the representation of a CRSM without deterministic component, see Theorem \ref{thm-representation}. The data point for the $m$-th individual is modelled as $X_{m}:=\sum_{k=1}^{K_{m}}x_{m,k}\delta_{\psi_{k}}$,
where $x_{m,k}$ represents the degree to which the $m$-th data point belongs to the trait $\psi_{k}$.

This model has been used in many applications when the $\theta_{k}$s take only positive values, for example in topic modeling (see \textit{e.g.}~\citet{Bro1,Bro2,topic-modelling}). In this section we explore the case of $\theta_{k}$ being real-valued, which allows for modeling of new classes of phenomena. For example, this is particularly important for sentiment analysis, also known as opinion mining, see the monograph \citet{liu}. Indeed, a basic task in sentiment analysis is classifying the \textit{polarity} of a given text in a document, namely whether the expressed opinion is positive, negative, or neutral.

From a formal point of view $\Theta$ and $X_{m}$ are defined as CRSMs. In particular, for the data $X_{m}$, we let $x_{m,k}$ be drawn according to some discrete distribution $H$ with probability mass function $h$ that takes $\theta_{k}$ as a parameter and has support on $\mathbb{Z}$, that is $x_{m,k}|\theta_k\stackrel{iid}{\sim}h(x|\theta_{k})$, independently across $m$ and $k$. We assume that $X_{1},...,X_{m}$ are i.i.d.~conditional on $\Theta$. The following assumptions for $\Theta$ and $X_{m}$ are the real-valued version of the assumptions in \citet{Bro1}:
\\\\ \textit{Assumption A00}: the atomless component of $\Theta$ has characteristic pair $(\gamma,\mu)$ s.t.~$\gamma=0$ and $\mu(d\theta\times d\psi)=\nu(d\theta)\cdot G(d\psi)$, where $\nu$ is any $\sigma$-finite measure on $\mathbb{R}\setminus\{0\}$ and $G$ is a proper distribution on $\Psi$ with no atoms. 
\\ \textit{Assumptions A0, A1, and A2}: $\Theta$ has finitely many fixed atoms, $\nu(\mathbb{R}\setminus\{0\})=\infty$, and $\sum_{x\in\mathbb{Z}\setminus\{0\}}\int_{\mathbb{R}\setminus\{0\}}h(x|\theta)\nu(d\theta)$ $<\infty$, respectively.
\\\\
We remark that by Assumption A00 we have that the location of the non-fixed atoms $\psi$ and the weights $\theta_{k}$ are stochastically independent. We call $\nu$ the \textit{weight measure} of $\Theta$. We call $\theta_{k}$ the \textit{weights} for consistency with the literature, even though here they can take negative values. Moreover, the assumption that $H$ has a point mass on $\mathbb{R}\setminus\{0\}$ and the assumptions A0, A1 and A2 come from a modeling need. By A0 we are assuming that we initially look at most finitely many investments, by A1 that there are possibly countably infinite investments, and by A2 that the amount of information from finitely represented data is finite, because by A2 the number of non-fixed atoms is finite, \textit{e.g.}~the competitor's portfolio has at most finitely many investments.

In the next result we show explicit formulations for the posterior distribution $\Theta|X$. We can write $\Theta=\sum_{k=1}^{K}\theta_{k}\delta_{\psi_{k}}$, where $K=K_{fix}+K_{ord}$. We denote the fixed component of $\Theta$ by $\Theta_{fix}=\sum_{k=1}^{K_{fix}}\theta_{fix,k}\delta_{\psi_{fix,k}}$ and the law of $\theta_{fix,k}$ by $F_{fix,k}:=\mathcal{L}(\Theta(\{\psi_{fix,k}\}))$.
\begin{proposition}\label{pro-Bay}
	Let $\Theta$ be a CRSM satisfying $A00$, $A0$, $A1$ and $A2$. Write $\Theta=\sum_{k=1}^{K}\theta_{k}\delta_{\psi_{k}}$, and let $X_{1}, . . . , X_{m}$ be generated i.i.d.~conditional on $\Theta$ according to $X_{1}:=\sum_{k=1}^{K}x_{1,k}\delta_{\psi_{k}}$ with $x_{1,k}|\theta_k\stackrel{iid}{\sim} h(\cdot|\theta_{k})$. Let $\Theta_{post}$ be a random signed measure with the distribution of $\Theta|X_1,...,X_m$. Then, $\Theta_{post}$ is a CRSM with three parts.
	
	\textnormal{1.} For each $k\in [K_{fix}]$, $\Theta_{post}$ has a fixed atom at $\psi_{fix,k}$ with weight $\theta_{post,fix,k}$ distributed according to $F_{post,fix,k} (d\theta)\propto F_{fix,k} (d\theta)\prod_{j=1}^{m}h(x_{fix,j,k}|\theta)$.
	
	\textnormal{2.} Let $\{\psi_{new,k}: k \in [K_{new} ]\}$ be the union of atom locations of $X_1,...,X_m$ minus the
	fixed atoms locations in the prior of $\Theta$. $K_{new}$ is finite. Let $x_{new,j,k}$ be the weight of the atom in $X_{j}$ located at $\psi_{new,k}$, for some $j=1,...,m$. Then, $\Theta_{post}$ has a fixed atom at $x_{new,k}$ with	random weight $\theta_{post,new,k}$, whose distribution $F_{post,new,k}(d\theta)\varpropto\nu(d\theta)\prod_{j=1}^{m}h(x_{new,j,k}|\theta)$.
	
	\textnormal{3.} The ordinary component of $\Theta_{post}$ has weight measure $\nu_{post}(d\theta):=\nu(d\theta)h(0|\theta)^m$.
\end{proposition}

We can generalize this framework (and the one of \citet{Bro1}) by modeling the fixed atoms differently from the other atoms and by using continuous distributions. Instead of $X=\sum_{k=1}^{K}x_{k}\delta_{\psi_{k}}$ with $x_{k}|\theta_k\stackrel{iid}{\sim} h(x|\theta_{k})$ we can model $X$ by $X=\sum_{k=1}^{K}x_{k}\delta_{\psi_{k}}$ with $x_{k}|\theta_k\stackrel{iid}{\sim} H_{fix}(dx|\theta_{k})$ if $\psi_{k}$ is a fixed atom and $x_{k}|\theta_k\stackrel{iid}{\sim} H_{ord}(dx|\theta_{k})$ otherwise, for some distribution $H_{fix}$ continuous and some distribution $H_{ord}$ continuous except for a point mass at zero. Denote by $f_{fix}$, $h_{fix}$, $h_{ord}$ the density of $F_{fix}$, $H_{fix}$ and $H_{ord}$ (on $\mathbb{R}\setminus\{0\}$), respectively. We have the equivalent of Assumption A2 in this setting:
\\\\ \textit{Assumption A2'}: $\int_{\mathbb{R}\setminus\{0\}}\int_{\mathbb{R}\setminus\{0\}}h_{ord}(x|\theta)dx \nu(d\theta)$ $<\infty$.
\begin{proposition}\label{pro-Bay-2}
	Let $\Theta$ be a CRSM satisfying $A00$, $A0$, $A1$ and $A2'$. Assume that $\nu(d\theta)=g(\theta)d\theta$ for some function $g:\mathbb{R}\setminus\{0\}\to(0,\infty)$. Let $X_1,...,X_m$ be generated i.i.d.~conditional on $\Theta$ according to $X_1=\sum_{k=1}^{K}x_{k}\delta_{\psi_{k}}$ with $x_{k}|\theta_k\stackrel{iid}{\sim} H_{fix}(dx|\theta_{k})$ if $\psi_{k}$ is a fixed atom and $x_{k}|\theta_k\stackrel{iid}{\sim} H_{ord}(dx|\theta_{k})$ otherwise. Then, $\Theta_{post}$ is a CRSM with three parts.
	
	\textnormal{1.} For each $k\in [K_{fix}]$, $\Theta_{post}$ has a fixed atom at $\psi_{fix,k}$ with weight $\theta_{post,fix,k}$ with density $f_{post,fix,k} (\theta)\propto f_{fix,k} (\theta)\prod_{j=1}^m h_{fix}(x_{fix,j,k}|\theta)$.
	
	\textnormal{2.} For each $k\in [K_{new}]$, $\Theta_{post}$ has a fixed atom at $\psi_{new,k}$ with weight $\theta_{post,new,k}$, whose density is proportional to $g(\theta)\prod_{j=1}^m h_{ord}(x_{new,j,k}|\theta)$.
	
	\textnormal{3.} The ordinary component of $\Theta_{post}$ has weight measure $\nu_{post}(d\theta):=\nu(d\theta)H_{ord}(\{0\}|\theta)^m$.
\end{proposition}

\begin{example}[Gaussian]
	Let $\Theta$ have finitely many fixed atoms. Let the density of the random weights of the fixed atoms of $\Theta$ be Gaussian, that is
	\begin{equation*}
		f_{fix,k}(\theta)=\frac{1}{\sqrt{2\pi\sigma_{fix,k}^2}}e^{-(\mu_{fix,k}-\theta)^{2}/ 2\sigma_{fix,k}^{2}}.
	\end{equation*}
	Let $\nu$ be given by $\nu(d\theta)=|\theta|^{\alpha-2}d\theta$ for some $\alpha\in(0,1)$. Let $X$ as in Proposition \ref{pro-Bay-2}. In particular, let $H_{ord}(\{0\}|\theta)=1-|\theta|^{2-\alpha}e^{-\theta^2}$. We remark that instead of $e^{-\theta^2}$ we can consider any integrable function $l(\theta)$ such that $l(\theta)=o(|\theta|^{-2})$ and that $|\theta|^{2-\alpha}l(\theta)\in[0,1]$.
	Further, let the density of the random weights $x_{k}$ conditional on $\theta_{k}$ be given by
	\begin{align*}
		h_{ord}(x|\theta_{k})&=(1-H(\{0\}|\theta_{k}))\frac{1}{\sqrt{2\pi\sigma^2}}e^{-(x-\theta_{k})^{2}/ 2\sigma^{2}},\text{	when $x\neq0$ and $\psi_k$ is fixed, and}\\
		h_{fix}(x|\theta_{k})&=\frac{1}{\sqrt{2\pi\sigma^2}}e^{-(x-\theta_{k})^{2}/ 2\sigma^{2}},\text{ when $\psi_k$ is not fixed.}
	\end{align*}
	Then, assumption A0 is satisfied and A1 is also satisfied because $\int_{\mathbb{R}\setminus\{0\}}|\theta|^{\alpha-2}d\theta=\infty$. We remark that $\nu$ also satisfies the underlying condition $\int_{\mathbb{R}\setminus\{0\}}1\wedge|\theta|\nu(d\theta)<\infty$ for any $\alpha\in(0,1)$. Further, A2' is satisfied because by Tonelli's theorem
	\begin{equation*}
		\int_{\mathbb{R}\setminus\{0\}}\int_{\mathbb{R}\setminus\{0\}}h_{ord}(x|\theta)\nu(d\theta)dx=\int_{\mathbb{R}\setminus\{0\}}1-H_{ord}(\{0\}|\theta)\nu(d\theta)=\int_{\mathbb{R}\setminus\{0\}}e^{-\theta^2}d\theta<\infty.
	\end{equation*}
	Then, Proposition \ref{pro-Bay-2} states that the posterior $\Theta_{post}$ has a fixed component where the densities for the random weights of the prior fixed atoms are given by
	\begin{equation*}
		f_{post,fix,k}(\theta)\varpropto \frac{|\theta|^{2-\alpha}e^{-\theta^2}}{\sqrt{2\pi\sigma^2}}\frac{1}{\sqrt{2\pi\sigma_{fix,k}^2}}e^{-(\mu_{fix,k}-\theta)^{2}/ 2\sigma_{fix,k}^{2}}e^{-(x_{fix,k}-\theta)^{2}/ 2\sigma^{2}},
	\end{equation*}
	while the densities for the random weights of the new fixed atoms are given by
	\begin{equation*}
		f_{post,new,k}(\theta)\varpropto \frac{e^{-\theta^2}}{\sqrt{2\pi\sigma^2}}e^{-(x_{fix,k}-\theta)^{2}/ 2\sigma^{2}}.
	\end{equation*}
	The ordinary component of $\Theta_{post}$ has weight measure $\nu_{post}(d\theta)=(1-|\theta|^{2-\alpha}e^{-\theta^2}) |\theta|^{\alpha-2}d\theta$. Finally, it is possible to check that $\Theta_{post}$ has finitely many fixed atoms, that $\nu_{post}(d\theta)(\mathbb{R}\setminus\{0\})=\infty$, that $\int_{\mathbb{R}\setminus\{0\}}1\wedge|\theta|\nu_{post}(d\theta)<\infty$ , and that
	\begin{equation*}
		\int_{\mathbb{R}\setminus\{0\}}\int_{\mathbb{R}\setminus\{0\}}h_{ord}(x|\theta)\nu_{post}(d\theta)dx=\int_{\mathbb{R}\setminus\{0\}}e^{-\theta^2}(1-|\theta|^{2-\alpha}e^{-\theta^2})d\theta<\infty.
	\end{equation*}
\end{example}

\subsection{Sparse signed graphs}\label{Sec-graph}
A random signed graph is a random graph in which each edge (also known as link) is associated with a positive or negative sign. Models for nodes interacting over such random signed graphs arise from various biological, social, and economic systems, see \citet{Siam} for a review. In physics, signed graphs are a natural context for the nonferromagnetic Ising model, which is applied to the study of spin glasses (\citet{Ising}). Recently, there has been a surge of interest in graph neural networks triggered by the work \citet{Derr}. 

In this section we follow the approach of \citet{Caron-and-Fox} who considered the random graph as a point process. In particular, we generalize it and consider an undirected random signed graph as a signed point process:
\begin{equation}\label{Z}
	Z=\sum_{i=1}^\infty\sum_{j=1}^\infty z_{ij}\delta_{(\theta_i,\theta_j)}
\end{equation}
with $z_{ij}=z_{ji}\in\{-1,0,1\}$. Here, $z_{ij}=z_{ji}=1$ indicates a positive undirected edge between nodes $\theta_i$ and $\theta_j$ and $z_{ij}=z_{ji}=-1$ a negative one. In social networks, $z_{ij}$ may indicate the presence or absence of a connection between $\theta_i$ and $\theta_j$, and if a connection exists, it also specifies whether it is a positive or a negative one.  We introduce a collection of per-node parameters $w=\{w_i\}$ and specify link probabilities via
\begin{align*}
	\mathbb{P}(z_{ij}=1|w)=\begin{cases}
		1-\exp(-2w_i w_j) & \textnormal{if $i\neq j$, $w_i>0$ and $w_j>0$}\\	1-\exp(-w_i^2) & \textnormal{if $i= j$ and $w_i>0$,}
	\end{cases}\\
	\mathbb{P}(z_{ij}=-1|w)=\begin{cases}
		1-\exp(-2w_i w_j) & \textnormal{if $i\neq j$, $w_i<0$ and $w_j<0$}\\	1-\exp(-w_i^2) & \textnormal{if $i= j$ and $w_i<0$}
	\end{cases}
\end{align*}
Our generative model jointly specifies $(w_{i},\theta_i)_{i\in\mathbb{N}}$, where $w_i\in\mathbb{R}\setminus\{0\}$ and $\theta_{i}\in[0,\infty)$, using the CRSM without fixed atoms and deterministic component:
\begin{equation*}
	W=\sum_{i=1}^{\infty}w_i\delta_{\theta_i}.
\end{equation*}

As we have seen in Theorem \ref{thm-representation}, $W$ can be uniquely equivalently written as $W=\int_{\mathbb{R}\setminus\{0\}}x\Psi(B\times dx)$, $B\in\hat{\mathcal{S}}$, for some Poisson process on $S\times(\mathbb{R}\setminus\{0\})$ with intensity measure $F$ satisfying 	$\int_{\mathbb{R}\setminus\{0\}}1\wedge|x| F(B\times dx)<\infty$, for every $B\in\hat{\mathcal{S}}$, and $F(\{s\}\times (\mathbb{R}\setminus\{0\}))=0$ for every $s\in S$. Similarly as in \citet{Caron-and-Fox}, we assume that $F(dw,d\theta)=\rho(dw)\lambda(d\theta)$, where $\lambda$ is the Lebesgue measure on $[0,\infty)$ and $\rho$ is a L\'{e}vy measure on $\mathbb{R}\setminus\{0\}$. Since $W$ is a CRSM without fixed atoms, by Corollary \ref{co-two-CRSM} its Jordan decomposition are two independent CRMs, which we denote by $W_+$ and $W_-$.

Now, consider two atomic measures $D_+$ and $D_-$ on $[0,\infty)^2$ generated from a Poisson process
with intensity given by the product measure $W_+\times W_+$ and $W_-\times W_-$, respectively, that is $D_+|W_+\sim PP(W_+\times W_+)$ and $D_-|W_-\sim PP(W_-\times W_-)$, where we allow the mean measure to have atoms. 

We now make a crucial assumption: we assume that $D_+$ and $D_-$ are independent. This assumption is made for clarity of exposition as it allows to show how our theoretical results immediately generalise the model and results of \citet{Caron-and-Fox}. On the other hand, this assumption implies independence between the positive and negative edges, limiting the range of empirical applications.

The undirected signed graph model is viewed as a transformation of a directed integer-weighted graph, or multigraph, which we represent as an atomic signed measure $D$ on $[0,\infty)^2$ given by the difference of $D_+$ and $D_-$:
\begin{equation*}
	D=D_+-D_-=\sum_{i=1}^\infty\sum_{j=1}^\infty n_{ij}\delta_{(\theta_i,\theta_j)}.
\end{equation*}
Thus, informally, $n_{ij}\sim Poisson (w_i w_j)$ when $w_i,w_j>0$, $-n_{ij}\sim Poisson (w_i w_j)$  when $w_i,w_j<0$, and $n_{ij}=0$ otherwise. Note that $n_{ij}\stackrel{d}{=}n_{ji}$ and they have the same sign almost surely. Then, by setting $z_{ij}=sign(n_{ij}+n_{ji})\min(|n_{ij}+n_{ji}|,1)$, where $sign(\cdot)$ is the sign function, we obtain the undirected signed graph model $Z$ in $(\ref{Z})$ and the resulting hierarchical model is
\begin{equation}\label{hierarchical}
	\begin{cases}
		W=\sum_{i=1}^{\infty}w_i\delta_{\theta_i}, \quad\quad\quad\quad\quad\quad\quad\quad\quad\quad\,\,\,\,\, W\sim CRSM(\rho,\lambda)\\
		D=\sum_{i=1}^\infty\sum_{j=1}^\infty n_{ij}\delta_{(\theta_i,\theta_j)}=D_+-D_-, \quad\,  D_\pm\sim PP(W_\pm\times W_\pm)\\
		Z=\sum_{i=1}^\infty\sum_{j=1}^\infty sign(n_{ij}+n_{ji})\min(|n_{ij}+n_{ji}|,1)\delta_{(\theta_i,\theta_j)},
	\end{cases}
\end{equation}
where $D_\pm\sim PP(W_\pm\times W_\pm)$ stands for both $D_+\sim PP(W_+\times W_+)$ and $D_-\sim PP(W_-\times W_-)$. We can naturally generalize the definition of exchangeable random measures to the signed case.
\begin{definition}
	[exchangeable random signed measure] A random signed measure $\xi$ on $\mathbb{R}^d$ is separate exchangeable if and only if 
	\begin{equation*}
		(	\xi(A_{i_1}\times\cdots\times A_{i_d}))\stackrel{d}{=}(	\xi(A_{\pi_1(i_1)}\times\cdots\times A_{\pi_d(i_d)}))\quad \textnormal{for $(i_1,...,i_d)\in\mathbb{N}^d$}
	\end{equation*}
	for any permutations $\pi_1,...,\pi_d$ of $\mathbb{N}$ and any intervals $A_j=[h(1-j),hj]$ with $j\in\mathbb{N}$ and $h>0$. When $\pi_1=...=\pi_d$, $\xi$ is called jointly exchangeable.
\end{definition}
\begin{proposition}\label{pro-exchange}
	For any $W\sim CRSM(\rho,\lambda)$, the point process $Z$ defined by equation $(\ref{hierarchical})$ is jointly exchangeable.
\end{proposition} 

We are now ready to discuss the sparsity properties of our random signed graph model. Let $\alpha>0$ and let $Z_\alpha$ the restriction of $Z$ on the square $[0,\alpha]^2$. Let $N_\alpha$ be the number of observed nodes, that is the number of nodes with degree at least one, and let $N_\alpha^{(e)}$ be the number of edges in this restricted network.  Precisely,
\begin{align*}
	N_\alpha&=\textnormal{card}\{\theta_i\in[0,\alpha]:Z(\{\theta_i\}\times[0,\alpha])\neq 0\},\\
	N_\alpha^{(e)}&=|Z|(\{(x,y)\in[0,\infty)^2: 0\leq x\leq y\leq\alpha).
\end{align*}

\begin{proposition}\label{pro-sparse}
	Consider the point process $Z$ in $(\ref{hierarchical})$. Assume that $\int_{\mathbb{R}\setminus\{0\}}|w|\rho(dw)<\infty$. If the CRSM W has finite activity, \textit{i.e.}~$\int_{\mathbb{R}\setminus\{0\}}\rho(dw)<\infty$, then the number of edges scales quadratically with the number of observed nodes, that is $N_{\alpha}^{(e)}=\Theta(N_{\alpha}^2)$, almost surely as $\alpha\to\infty$, implying that the graph is dense. If the CRSM $W$ has infinite activity, \textit{i.e.}~$\int_{\mathbb{R}\setminus\{0\}}\rho(dw)=\infty$, then the number of edges scales subquadratically with the number of observed nodes, that is $N_{\alpha}^{(e)}=o(N_{\alpha}^2)$, almost surely as $\alpha\to\infty$, implying that the graph is sparse.
\end{proposition}

The model in \citet{Caron-and-Fox} has been extended in various directions (\textit{e.g.}~\citet{CaronMiscouridou} and \citet{RicciGuindaniS22}). We remark that the same can be done for our random signed graph model $(\ref{hierarchical})$.

\subsection{Mean function estimation in nonparametric regression}\label{Sec-Mean}
One of the Bayesian nonparametric approaches to the sparse regression problem based on mixtures relies on the idea that the regression function can be modeled as
\begin{equation*}
	f(\cdot)=\int_SK(x;\cdot)\xi(dx),\quad \xi\sim\Gamma,
\end{equation*}
where $K:S\times\mathbb{R}^d\to\mathbb{R}$ is a jointly measurable kernel function, and $\Gamma$ is a prior distribution on $\mathsf{M}_{S}$, thus $\xi$ is a random measure on $S$. This framework emerged in the works \citet{DeBlasiPeccatiPrunster,GhosalVanDerVaart2007,IshwaranJames,LijoiNipoti,LoWeng1989,PeccatiPrunster09}. The case of $\xi$ being real valued and with independent increments is explored in \citet{Naulet} and in \citet{Wolpert}. 

Our results open the possibility to investigate the general case of $\xi$ being a random signed measures. For example, we can leverage on Lemma \ref{lem-Jordan}, the Jordan decomposition of $\xi$, to immediately obtain the integrability and measurability conditions for $K$. Indeed, it is possible to see that $K$ is integrable and measurable if it is integrable and measurable for $\xi^+$ and for $\xi^-$ (or simply for $|\xi|$).

Even in the case of $\xi$ having independent increments our results provide a contribution. In \citet{Wolpert} the authors focus on independently scattered random measures (i.s.r.m.) on a complete separable metric space. In the uncompensated case, they focus on a specific class of CRSMs, namely the ones which are constructed starting from a Poisson random measure $\mathcal{N}$ and are of the form $\xi(A)=\int_{\mathbb{R}}\int_A \beta \mathcal{N}(d\beta ds)$ where $A\in\hat{\mathcal{S}}$. Thanks to our results, we can immediately extend all the results in \citet{Wolpert} for the uncompensated case to hold for every CRSMs. Indeed, by Theorem \ref{thm-representation} any CRSMs has an almost sure representation of the form $\xi(A)=\int_{\mathbb{R}}\int_A \beta \mathcal{N}(d\beta ds)$.

Moreover, in \citet{Naulet} the authors focus on the case of $\xi$ being a symmetric Gamma random measure. They considered it as a pre-kernel, precisely as a map from $\Omega\times \mathcal{S}\mapsto \mathbb{R}\pm\{\infty\}$ such that, for fixed $\omega\in\Omega$, $\xi(\omega,\cdot)$ is a signed measure on $\mathcal{S}$ and, for fixed $B\in\mathcal{S}$, $\eta(\cdot, B)$ is a $\mathcal{F}$-measurable function. The authors state that the existence of such a random measure is given by \citet{RajRos}. However, in \citet{RajRos} the authors focus on i.s.r.m, which dot not necessarily have a kernel-like representation. Therefore, the existence (and uniqueness) of such random signed measure is established solely by Lemma \ref{lem-kernel}.

We also remark that in \citet{Naulet} the authors mentioned that finite dimensional distributions convergence is equivalent to convergence in distribution for random signed measures. They rely on Theorem 4.2 in \citet{Kallenberg0}. However, Theorem 4.2 in \citet{Kallenberg0} relies on the fact that the space of measures endowed with the vague topology is Polish (see 15.7.7 in \citet{Kallenberg0}), but it is known that the vague topology is not metrizable on the space of signed Radon measures, namely on $\mathsf{M}_{S}$ (as they also correctly state it in Section S1.1). 

Thus, the equivalence between finite dimensional distributions and convergence in distribution for random signed measures remains an important open problem with immediate applications.

\section{Conclusion}

In this paper, we introduced random signed measures, a novel class of statistical models, which naturally generalize random measures. We studied their properties, explored various examples, and investigated their role in Bayesian analysis through three applications. These applications are primarily based on existing models for random measures, demonstrating how these established models can be extended to a broader, real-valued framework. Nonetheless, we anticipate the emergence of new application domains, particularly those involving the Gaussian random measure and the Skellam point process. Moreover, in this paper we laid down the foundations of the theory of random signed measures, but many important theoretical questions remain open, including the equivalence between finite-dimensional distributions and convergence in distribution for random signed measures.
We leave these exciting applied and theoretical directions for further research.

\section*{Appendix}
\subsection*{Preliminaries}
There are some concepts that we use in the proofs that we do not introduce here, like dissection system, monotone class, rings, etc... These concepts are well explained in the two main references used in this paper, namely \cite{Kallenberg2} and \cite{Daley1,Daley}. In our proofs we use the following results of \cite{Pass}.
\begin{lemma}[Lemma 2.15 in \cite{Pass}]\label{215}
	Let $X$ be an arbitrary non-empty set and let $\mathcal{R}$ be a $\delta$-ring. Then, for every $E\in\mathcal{R}$ we have that $\{E\cap B:B\in\mathcal{R}\}$ is a $\sigma$-algebra.
\end{lemma}
\begin{lemma}[Lemma 2.16 in \cite{Pass}]\label{lem-vecchio-2}
	Let $X$ an arbitrary non-empty set and let $\mathcal{R}$ a $\delta$-ring. Let $\mu$ be a (possibly infinite) signed measure on $\mathcal{R}$. Then, there exist two unique measures $\mu^{+}$ and $\mu^{-}$ on $\mathcal{R}$ such that $\mu=\mu^{+}-\mu^{-}$ and that on any $A\in\mathcal{R}$ they are mutually singular.
\end{lemma}
\begin{lemma}[Lemma 2.18 in \cite{Pass}]\label{lem-vecchio-3}
	Let $X$ an arbitrary non-empty set and let $\mathcal{R}$ a $\delta$-ring. Let $\mu$ be a $\sigma$-finite signed measure on $\mathcal{R}$ (namely there exists a sequence $S_{1},S_{2},...\in\mathcal{R}$ s.t.~$X=\cup_{n=1}^{\infty}S_{n}$ and that $-\infty<\mu(S_{n})<\infty$ for every $n\in\mathbb{N}$). Then $\mu^{+}$ and $\mu^{-}$ can be uniquely extended to two $\sigma$-finite measures on $(X,\sigma(\mathcal{R}))$.
\end{lemma}
From the above it is possible to see that $|\mu|:=\mu^{+}+\mu^{-}$ is the total variation of $\mu$. Moreover we note that $\hat{\mathcal{S}}$ is closed under countable intersections (see page 19 in \cite{Kallenberg2}), hence $\hat{\mathcal{S}}$ is a $\delta$-ring. Further recall that every $\sigma$-ring is a $\delta$-ring.

\subsection*{Proof of Section \ref{Sec-measure}}

\begin{proof}[Proof of Lemma \ref{lem-measurable-vector-space}]
	It is immediate to see that $\mathsf{M}_{S}$ is a vector space. Moreover, it is sufficient to prove measurability with respect to the generator of $\mathcal{B}_{\mathsf{M}_S}$. Observe that the family of sets of $\{\mu\in\mathsf{M}_S:\mu(B)\in A\}$ where $B\in\hat{\mathcal{S}}$ and $A\in\mathcal{B}(\mathbb{R})$ is also a generator of $\mathcal{B}_{\mathsf{M}_S}$. Thus, consider any $B\in\hat{\mathcal{S}}$, $A\in\mathcal{B}(\mathbb{R})$, and $x\in\mathbb{R}$. We obtain the measurability of scalar multiplication using that $c\mu(B)\in A$ is equivalent to $\mu(B)\in \{x/c: x\in A\}(\in\mathcal{B}(\mathbb{R}))$, where $c\in\mathbb{R}\setminus0$. Indeed, 
	\begin{equation*}
		\{\eta\in\mathsf{M}_S:c\eta\in \{\mu\in\mathsf{M}_S:\mu(B)\in A\} \}=	\{\eta\in\mathsf{M}_S:c\eta(B)\in A\}
	\end{equation*}
	\begin{equation*}
		=	\{\eta\in\mathsf{M}_S:\eta(B)\in \{x/c: x\in A\} \}\in \mathcal{B}_{\mathsf{M}_S}.
	\end{equation*}
	The case $c=0$ is trivial. Concerning the measurability of addition, we have
	\begin{equation*}
		\{(\eta,\nu)\in\mathsf{M}_S\times\mathsf{M}_S:\eta+\nu\in \{\mu\in\mathsf{M}_S:\mu(B)<x\} \}		=\{(\eta,\nu)\in\mathsf{M}_S\times\mathsf{M}_S:\eta(B)+\nu(B)<x \}.
	\end{equation*}
	Hence, it remains to show that $\{(\eta,\nu)\in\mathsf{M}_S\times\mathsf{M}_S:\eta(B)+\nu(B)<x\}$ belongs to $\mathcal{B}_{\mathsf{M}_S}\otimes\mathcal{B}_{\mathsf{M}_S}$. Note that $\eta(B)+\nu(B)<x$ if and only if there exists a rational number $r$ such that $\eta(B)<r<x-\nu(B)$. Therefore,
	\begin{equation*}
		\{(\eta,\nu)\in\mathsf{M}_S\times\mathsf{M}_S:\eta(B)+\nu(B)<x\}
	\end{equation*}
	\begin{equation*}
		=\bigcup_{r\in\mathbb{Q}}\Big[ \{(\eta,\nu)\in\mathsf{M}_S\times\mathsf{M}_S:\eta(B)<r\}\cap \{(\eta,\nu)\in\mathsf{M}_S\times\mathsf{M}_S:\nu(B)<x+r\} \Big]
	\end{equation*}
	and since $ \{(\eta,\nu)\in\mathsf{M}_S\times\mathsf{M}_S:\eta(B)<r\}$ and $\{(\eta,\nu)\in\mathsf{M}_S\times\mathsf{M}_S:\nu(B)<x+r\}$ belong to $\mathcal{B}_{\mathsf{M}_S}\otimes\mathcal{B}_{\mathsf{M}_S}$, we obtain the result.
\end{proof}

\begin{proof}[Proof of Lemma \ref{lem-uncited-1}]
	Since $f=\mathbf{1}_B$, where $B\in\hat{\mathcal{S}}$, is an element of $\mathfrak{F}_{b}(\mathcal{S})$, we have that the $\sigma$-algebra generated by $\{\pi_f\}_{f\in\mathfrak{F}_{b}(\mathcal{S})}$ contains $\mathcal{B}_{\mathsf{M}_S}$. For the other direction, we have the following. Consider any $f\in\mathfrak{F}_{b}(\mathcal{S})$ and let $f_+(s)=\max(f(s),0)$ and $f_-(s)=\max(-f(s), 0)$, then $f=f_+-f_-$. By approximating $f_+$ and $f_-$ by nondecreasing sequences of simple functions and using the monotone convergence theorem we obtain the measurability of $\mu(f_+)$ and $\mu(f_-)$ and so of $\mu(f)$ for every $\mu\in \mathsf{M}_{S}$.
\end{proof}

\begin{proof}[Proof of Lemma \ref{lem-inclusion}]
	By definition we have that any locally finite measure on $S$ is a locally finite countably additive set function on $\hat{\mathcal{S}}$, hence $\mathsf{M}_{S}\subset\mathsf{M}_{S}$. Moreover, since $\mathcal{B}_{\mathcal{M}_S}$ is generated by all projection maps $\pi_{B}:\mu\mapsto\mu(B)$ where $B\in\hat{\mathcal{S}}$ and since $\mathsf{M}_{S}\subset\mathsf{M}_{S}$, we obtain that $\mathcal{B}_{\mathcal{M}_S}\subset \mathcal{B}_{\mathsf{M}_S}$. Let $B\in\hat{\mathcal{S}}$ and $A\in\mathcal{B}(\mathbb{R})$. Since
	\begin{equation*}
		\{\eta\in\mathcal{M}_S:\eta\in \{\mu\in\mathsf{M}_S:\mu(B)\in A\} \}=	\{\eta\in\mathcal{M}_S:\eta(B)\in A\}\in \mathcal{B}_{\mathcal{M}_S},
	\end{equation*}
	we obtain the measurability of the inclusion map. 
	
	For the last assertion we focus on the map $\mu\mapsto\mu_+$ because the same arguments apply to $\mu\mapsto\mu_-$. Let $\mathsf{M}_S^+$ be the space of all locally finite countably additive nonnegative set functions on $\hat{\mathcal{S}}$ and let $\mathcal{B}_{\mathsf{M}_S^+}$ be the $\sigma$-algebra in $\mathsf{M}_S^+$ generated by all projection maps $\pi_{B}:\mu\mapsto\mu(B)$, $B\in\hat{\mathcal{S}}$. By Lemma 1.3 in \cite{Kallenberg2} there exists a dissection system, which we denote by $\mathcal{R}$, that is the countable semi-ring of bounded Borel sets generating $\mathcal{S}$ (see \cite{Kallenberg2} page 16). Since $\mathcal{R}$ generates $\mathcal{S}$ and $\hat{\mathcal{S}}\subset\mathcal{S}$, for every $\epsilon>0$ and $D\in\hat{\mathcal{S}}$ there exists a set $C\in\mathcal{R}$ such that $\mu_+(D\Delta C)<\epsilon$ (where $D\Delta C:=(D\setminus C)\cup(C\setminus D)$). Then, for any $B\in\hat{\mathcal{S}}$ we have
	\begin{equation*}
		\mu_+(B)=\sup_{E\in \mathcal{S}, E\subset B}\mu(E)= \sup_{E\in \hat{\mathcal{S}}, E\subset B}\mu(E)=\sup_{E\in \mathcal{R}, E\subset B}\mu(E).
	\end{equation*}
	Then, for some $B\in\hat{\mathcal{S}}$ and $x\in\mathbb{R}$ we have 
	\begin{equation*}
		\{\eta\in\mathsf{M}_S:\eta_+\in \{\nu\in\mathsf{M}^+_S:\nu(B)>x\} \}=	\{\eta\in\mathsf{M}_S:\eta_+(B)>x\}
	\end{equation*}
	\begin{equation*}
		=\bigcup_{E\in \mathcal{R}, E\subset B}	\{\eta\in\mathsf{M}_S:\eta(E)>x\}
	\end{equation*}
	which is an element of $\mathsf{B}_{\mathsf{M}_S}$. Thus, the map that sends $\mu$ to $\mu_+$, which is a map from $(\mathsf{M}_{S},\mathcal{B}_{\mathsf{M}_S})$  to $(\mathsf{M}^+_{S},\mathcal{B}_{\mathsf{M}^+_S})$,  is measurable. Now, it remains to show that the map that sends $\mu_+$ to its extension $\mu_+$, which is a map from $(\mathsf{M}^+_{S},\mathcal{B}_{\mathsf{M}^+_S})$ to $(\mathcal{M}_{S},\mathcal{B}_{\mathcal{M}_S})$, is measurable. Since any locally finite nonnegative countably additive set function on $\hat{\mathcal{S}}$ has a unique extension to a locally finite measure on $\mathcal{S}$ by the Carath\'{e}odory's extension theorem and since any such measure is a unique locally countably additive set function on $\hat{\mathcal{S}}$, there is a one-to-one correspondence between the elements of $\mathsf{M}^+_S$ and of $\mathcal{M}_S$. Then, for any $B\in\hat{\mathcal{S}}$ and $A\in\mathcal{B}(\mathbb{R}_+)$ we have
	\begin{equation*}
		\{\eta\in\mathsf{M}^+_S:\eta\in \{\mu\in\mathcal{M}_S:\mu(B)\in A\} \}=	\{\eta\in\mathsf{M}^+_S:\eta(B)\in A\}\in \mathcal{B}_{\mathsf{M}^+_S}.
	\end{equation*}
	Hence, the map that sends $\mu_+$ to its extension $\mu_+$ is measurable. By composing these two maps we obtain that the map $\mu\mapsto\mu_+$ from $(\mathsf{M}_{S},\mathcal{B}_{\mathsf{M}_S})$ to $(\mathcal{M}_{S},\mathcal{B}_{\mathcal{M}_S})$ is measurable.
\end{proof}

\begin{lemma}\label{lem-dissection}
	The collection of all the finite unions of the elements of a dissection system of $S$ is a countable ring.
\end{lemma}
\begin{proof}
	By Lemma 1.3 in \cite{Kallenberg2} there exists a dissection system, which is a semi-ring of bounded Borel sets generating $\mathcal{S}$ (see \cite{Kallenberg2} pages 16 and 20). Moreover, the dissection system is a countable set because by the axiom of choice the union of countably many countable sets is countable. Moreover, by Lemma 1.2.14 in \cite{Bogachev} (see also page 19 in \cite{Kallenberg2}) the finite unions of sets of a semi-ring form a ring, which we call it $\mathcal{R}$. It remains to show that $\mathcal{R}$ is countable. Notice that $\mathcal{R}$ is the ring generated by the dissection system (see page 8 in \cite{Bogachev}). Then, by Theorem C on page 23 in \cite{Halmos} we obtain that $\mathcal{R}$ is countable.
\end{proof}

\begin{proof}[Proof of Lemma \ref{lem-countable-ring-existence}]
	By Lemma 1.3 in \cite{Kallenberg2} there exists a dissection system. Since a finite union of elements in a dissection system is a bounded Borel set, by Lemma \ref{lem-dissection} we obtain the result.
\end{proof}

\begin{proof}[Proof of Theorem \ref{thm-extension}]
	For brevity we call $\eta=(\eta(U))_{U\in\mathcal{R}}$ a random set function. The existence of $\mathcal{R}$ is proved in Lemma \ref{lem-countable-ring-existence}. Since $\mathcal{R}$ is countable, we have that the random set functions $\eta_+$ and $\eta_-$ on $ \mathcal{R}$, defined by
	\begin{equation*}
		\eta_+(A)=\sup_{E\in \mathcal{R}, E\subset A} \eta(E)\quad\textnormal{and}\quad 	\eta_-(A)=\sup_{E\in \mathcal{R}, E\subset A} -\eta(E),\quad\forall A\in \mathcal{R},
	\end{equation*}
	are well-defined. By point (iii) we immediately have that $\eta_+(A)<\infty$ and $\eta_-(A)<\infty$  almost surely for all $A\in \mathcal{R}$. Moreover, point (i) implies that $\eta_+$ and $\eta_-$ are also finitely additive in the sense of point (i). Indeed, let $A,B\in \mathcal{R}$ disjoint. We have
	\begin{equation*}
		\eta_+(A\cup B)=\sup_{E\in \mathcal{R}, E\subset A\cup B} \eta(E)=\sup_{E\in \mathcal{R}, E\subset A\cup B} \eta((E\cap A)\cup(E\cap B))
	\end{equation*}
	\begin{equation*}
		\stackrel{a.s.}{=}\sup_{E\in \mathcal{R}, E\subset A\cup B} \eta(E\cap A)+\eta(E\cap B)
	\end{equation*}
	\begin{equation*}
		=\sup_{E_1\in \mathcal{R}, E_1\subset A} \eta(E_1)+\sup_{E_2\in \mathcal{R}, E_2\subset B} \eta(E_2)=\eta_+(A)+\eta_+(B),
	\end{equation*}
	where we used that a ring is closed under finite intersections and finite unions, and so for any $E\subset A\cup B$ we have that $E\cap A$ and $ E\cap B$ belong to  $\mathcal{R}$, and for any $E_1\subset A$ and $E_2\subset B$ with $E_1,E_2\in \mathcal{R}$ we have that $E_1\cup E_2\in \mathcal{R}$. Note that we also use that $\mathcal{R}$ is countable and so the set $$\{\omega\in\Omega: \eta((E\cap A)\cup(E\cap B))(\omega)= \eta(E\cap A)(\omega)+\eta(E\cap B)(\omega)\textnormal{ for all $E\in \mathcal{R}$ s.t.~$E\subset A\cup B$}\}$$
	is of probability measure 1. The same arguments apply to $\eta_-$. 
	
	We now show that the finite additivity of $\eta_+$ together with points (i) and (ii) imply the countable additivity of $\eta_+$. First, we prove the countable additivity of $\eta$, in the sense that given $(B_k)_{k\in\mathbb{N}}$ a sequence of disjoints sets in $\mathcal{R}$ s.t.~$B=\bigcup_{k\in\mathbb{N}}B_k\in\mathcal{R}$ we have $\eta(B)\stackrel{a.s.}{=}\sum_{k\in\mathbb{N}}\eta(B_k)$. So, consider such sequence $(B_k)_{k\in\mathbb{N}}$ and let $A_n=\bigcup_{k=n}^\infty B_k$, for $n\in\mathbb{N}$. Since a ring is closed under finite unions and proper differences, $A_n\in\mathcal{R}$ for every $n\in\mathbb{N}$. Then, by point (i) we have that $\eta(B)\stackrel{a.s.}{=}\sum_{k=1}^{n-1}\eta(B_k) +\eta(A_n)$ and by point (ii) we have that $\eta(A_n)\stackrel{a.s.}{\to}0$ as $n\to\infty$. Thus, we have
	\begin{equation*}
		\eta(B)-\sum_{k=1}^{n-1}\eta(B_k)\stackrel{a.s.}{=}\eta(A_n)\stackrel{a.s.}{\to}0,\quad\textnormal{as $n\to\infty$},
	\end{equation*}
	and so $\eta(B)\stackrel{a.s.}{=}\sum_{k\in\mathbb{N}}\eta(B_k)$. This holds for any order of the elements of $(B_n)_{n\in\mathbb{N}}$. From this we obtain
	\begin{equation*}
		\eta^+\left(\bigcup_{k\in\mathbb{N}}B_k\right)=\sup_{E\in \mathcal{R}, E\subset \cup_{k\in\mathbb{N}}B_k} \eta(E)=\sup_{E\in \mathcal{R}, E\subset \cup_{k\in\mathbb{N}}B_k} \eta\left(\bigcup_{k\in\mathbb{N}}(E\cap B_k)\right)
	\end{equation*}
	\begin{equation*}
		\stackrel{a.s.}{=}\sup_{E\in \mathcal{R}, E\subset \cup_{k\in\mathbb{N}}B_k} \sum_{k\in\mathbb{N}}\eta(E\cap B_k)\leq \sum_{k\in\mathbb{N}} \sup_{E\in \mathcal{R}, E\subset \cup_{k\in\mathbb{N}}B_k} \eta(E\cap B_k)
	\end{equation*}
	\begin{equation*}
		=\sum_{k\in\mathbb{N}} \sup_{E\in \mathcal{R}, E\subset B_k} \eta(E)=\sum_{k\in\mathbb{N}}\eta_+(B_k)
	\end{equation*}
	where we used again that $\mathcal{R}$ is countable. So, we obtain $		\eta_+\left(\bigcup_{k\in\mathbb{N}}B_k\right)\leq \sum_{k\in\mathbb{N}}\eta_+(B_k)$ a.s.. On the other hand, by finite additivity we have that $		\eta_+\left(\bigcup_{k\in\mathbb{N}}B_k\right)\geq \sum_{k\in\mathbb{N}}\eta_+(B_k)$ a.s., indeed for all $n\in\mathbb{N}$
	\begin{equation*}
		\eta_+\left(\bigcup_{k\in\mathbb{N}}B_k\right)\stackrel{a.s.}{=}\eta_+\left(\bigcup_{k=1}^{n}B_k\right)+\eta_+\left(\bigcup_{k=n+1}^{\infty}B_k\right)\geq \eta_+\left(\bigcup_{k=1}^{n}B_k\right)\stackrel{a.s.}{=} \sum_{k=1}^n\eta_+(B_k).
	\end{equation*}
	Thus, we conclude that $\eta_+(\bigcup_{k\in\mathbb{N}}B_k)\stackrel{a.s.}{=} \sum_{k\in\mathbb{N}}\eta_+(B_k)$. Hence, $\eta_+$ and (using the same arguments) $\eta_-$ are countably additive.
	
	For any $C_n\downarrow\emptyset$ along $\mathcal{R}$, we now show that $\eta_+(C_n)\stackrel{a.s.}{\to}0$. Let $D_n=C_{n}\setminus C_{n+1}$. Then, $D_j\cap D_i=\emptyset$ for every $i,j\in\mathbb{N}$ with $i\neq j$, and $C_n=\bigcup_{k=n}^{\infty}D_k$ for every $n\in\mathbb{N}$. By countable additivity of $\eta_+$ we have $\eta_+(C_n)\stackrel{a.s.}{=}\sum_{k=n}^{\infty}\eta_+(D_k)$. Since by (iii) we know that $\sum_{k=n}^{\infty}\eta_+(D_k)<\infty$ a.s., we have that $\sum_{k=n}^{\infty}\eta_+(D_k)\stackrel{a.s.}{\to}0$ as $n\to\infty$. Hence, we conclude that $\eta_+(C_n)\stackrel{a.s.}{\to}0$ as $n\to\infty$. The same holds for $\eta_-$.
	
	Therefore, $\eta_+$ and $\eta_-$ satisfy the conditions of Theorem 2.15 in \cite{Kallenberg2} and so there exist two a.s~unique random measures $\xi_+$ and $\xi_-$ on $(S,\mathcal{S})$ such that $\xi_+(A)\stackrel{a.s.}{=}\eta_+(A)$ and $\xi_-(A)\stackrel{a.s.}{=}\eta_-(A)$ for all $A\in\mathcal{R}$. Hence, since $\eta(A)\stackrel{a.s.}{=}\eta_+(A)-\eta_-(A)$, by Lemmas \ref{lem-measurable-vector-space} and \ref{lem-inclusion} we obtain the result. The necessity is trivial. In particular see Theorem 2.15 in \cite{Kallenberg2} for the necessity of conditions (i) and (ii) and see Example 9.1(f) in \cite{Daley} for the necessity of condition (iii).
\end{proof}

\begin{example}[Counterexample of Theorem 10 in \cite{Jacob1995}] In this example we provide a counterexample to Theorem 10 (and Corollary 11) in \cite{Jacob1995}, in particular we show that condition (M2) is not sufficient for their results and so also for our Theorem \ref{thm-extension} and Corollary \ref{co-extension}, even for deterministic measures.
	
	Let $S$ be a compact Hausdorff second countable topological space. Let $\mathcal{B}$ a countable basis of the topology of $S$, which we can assume to be closed under finite unions and intersections, and let $\mathcal{A}$ be the ring induced by $\mathcal{B}$. We assume that $S\in\mathcal{B}$, so that $\mathcal{A}$ is an algebra. Let $\mu$ be a countably additive set function on $\mathcal{A}$ with values in $\mathbb{R}$ and $\mu(\emptyset)=0$ (this is called a real measure in the notation of \cite{RaoRao93}).
	
	By Theorem 2.5.3 in \cite{RaoRao93} we know that $\mu$ has a Jordan decomposition composed by two positive measures $\mu^+$ and $\mu^-$ on $\mathcal{A}$. Denote by $|\mu|:=\mu^++\mu^-$, which is a positive measure on $\mathcal{A}$. From the same theorem we know that $|\mu|$ is bounded if and only if $\mu$ is bounded.
	
	Assume that $|\mu|(A_n)\to0$ for every sequence of sets $(A_n)_{n\in\mathbb{N}}$ such that $A_n\in\mathcal{A}$ (for every $n\in\mathbb{N}$) and $A_n\downarrow\emptyset$.
	
	Question: Can we conclude that $|\mu|$ is bounded?
	
	No, because of the following counterexample. Let $S = \mathbb{N} \cup \{\infty\}$ be the one-point compactification of the naturals. Let $\mathcal{B}$ consists of finite subsets of $\mathbb{N}$ as well as their complements, which is indeed a countable basis of topology for $S$. Note that $\mathcal{B}$ is already closed under finite unions, finite intersections, and taking complements, so $\mathcal{A} = \mathcal{B}$. Let $\mu$ be defined by $\mu(A) = |A|$ if $A$ is finite and $\mu(A) = -|S \setminus A|$ if $A$ is co-finite. Then $\mu$ is a real measure in the sense of \cite{RaoRao93}. Observe that there is no countable collection of disjoint subsets of $\mathcal{A}$ whose union is in $\mathcal{A}$ unless all but finitely many of them are empty, since all finite sets in $\mathcal{A}$ do not contain $\infty$ but all infinite sets in $\mathcal{A}$ are co-finite and contain $\infty$. So, countable additivity just reduces to finite additivity, which is easy to verify.
	
	Further, observe that $\mu^+$ is simply the counting measure on $\mathbb{N}$ and $\mu^-$ is $0$ on finite sets and $\infty$ on co-finite sets, so $|\mu|$ is the counting measure on $S$. Therefore, $\mu$ and $|\mu|$ are not bounded. Now, let $A_n \downarrow \varnothing$. Since all infinite sets in $\mathcal{A}$ contain $\infty$, we have that if all the $A_n$'s were infinite then $\bigcap_{n\in\mathbb{N}} A_n=\{\infty\}\neq\emptyset$. Thus, we must have $A_n$ are eventually finite, and thus eventually the empty set, from which we obtain that $|\mu|(A_n) \to 0$.
	
	Therefore, we conclude that, since $\mu$ is unbounded, $\mu\notin\mathsf{M}_S$ (that is $\mu$ is not a Radon measure). Thus, conditions (M1) and (M2) in Theorem 10 in \cite{Jacob1995} are not sufficient to ensure the existence of a random signed measure $\xi$ on $S$.
\end{example}

\begin{proof}[Proof of Lemma \ref{lem-uncited-2}]
	It is immediate from the definition of random measures and of random signed measures and from Lemmas \ref{lem-measurable-vector-space} and \ref{lem-inclusion}.
\end{proof}

\begin{proof}[Proof of Lemma \ref{lem-Jordan}]
	For every $\omega\in\Omega$ and $B\in\hat{\mathcal{S}}$ let $		\eta_+(\omega,B):=\sup_{E\in \mathcal{S}, E\subset B}\xi(\omega,E)$. Fix $\omega\in\Omega$. Since $\{E\in \mathcal{S}, E\subset B\}=\{E\in \hat{\mathcal{S}}, E\subset B\}$, by Lemmas \ref{lem-vecchio-2} and \ref{lem-vecchio-3} there exists a unique measure on $\mathcal{S}$ which extends $\eta_+(\omega,\cdot)$. We denote this extension by $\xi_+(\omega,\cdot)$. Now, we show that $\xi_+(\cdot,D)$ is measurable for every $D\in \mathcal{S}$. Let $\mathcal{R}$ be a countable ring of bounded Borel sets generating $ \mathcal{S}$, which exists thanks to Lemma \ref{lem-countable-ring-existence}. Since $\mathcal{R}$ generates $\mathcal{S}$ and $\hat{\mathcal{S}}\subset\mathcal{S}$, for every $\omega\in\Omega$, $\epsilon>0$, $D\in\hat{\mathcal{S}}$ there exists a set $C\in\mathcal{R}$ such that $\xi_+(\omega,D\Delta C)<\epsilon$ (where $D\Delta C:=(D\setminus C)\cup(C\setminus D)$). Then, for any $B\in\hat{\mathcal{S}}$ we have
	\begin{equation*}
		\xi_+(\omega,B)=\sup_{E\in \mathcal{S}, E\subset B}\xi(\omega,E)= \sup_{E\in \hat{\mathcal{S}}, E\subset B}\xi(\omega,E)=\sup_{E\in \mathcal{R}, E\subset B}\xi(\omega,E)
	\end{equation*}
	and since $\mathcal{R}$ is a countable set, $\xi_+(\cdot,B)$ is measurable. Then,  by Lemma 1.14 point (iii) in \cite{Kallenberg2} we obtain that $\xi_+(\cdot,D)$ is measurable for every $D\in \mathcal{S}$ and in particular that it is random measure. The same arguments apply to $\xi_-$. Thus, $\xi_+$ and $\xi_-$ are random measures and satisfy the stated properties.
\end{proof}

\begin{proof}[Proof of Corollary \ref{co-atom}]
	It follows from Lemma \ref{lem-Jordan} and the fact that $|\xi|$ has at most countably many fixed atoms (see Corollary 2.6 in \cite{Kallenberg2}).
\end{proof}

\begin{proof}[Proof of Corollary \ref{co-uncited-3}]
	It follows from Corollary 2.5 in \cite{Kallenberg2} and Lemma \ref{lem-Jordan}.
\end{proof}

\begin{proof}[Proof of Lemma \ref{lem-deterministic-representation}]
	By Lemma  \ref{lem-vecchio-3} we know that there exist two unique measures $\mu_+$ and $\mu_-$ such that $\mu=\mu_+-\mu_-$. By Lemma 1.6 in \cite{Kallenberg2} we obtain that the unique representation
	\begin{equation*}
		\mu_+=\alpha'_+\sum_{k\leq\kappa'_+}\beta'_{+,k}\delta_{\sigma'_{+,k}}\quad\text{and}\quad \mu_-=\alpha'_-+\sum_{k\leq\kappa'_-}\beta'_{-,k}\delta_{\sigma'_{-,k}}
	\end{equation*}
	where $\alpha'_+$, $\kappa'_+$, and all $(\beta'_{+,k},\sigma'_{+,k})$ can be chosen to be measurable functions of $\mu_+$ and $\alpha'_-$, $\kappa'_-$, and all $(\beta'_{-,k},\sigma'_{-,k})$ can be chosen to be measurable functions of $\mu_-$. Observe that here $\alpha'_+$, $\kappa'_+$, $(\beta'_{+,k},\sigma'_{+,k})$, $\alpha'_-$, $\kappa'_-$, and $(\beta'_{-,k},\sigma'_{-,k})$ all act on $\mathcal{M}_S$.
	
	Consider $\kappa'_+$. We can extend the domain of $\kappa'_+$ from $\mathcal{M}_S$ to $\mathsf{M}_S$ by defining a new map, which we call $\kappa_+$, such that $\kappa_+(\mu)=\kappa'_+(\mu_+)$. Since by Lemma \ref{lem-inclusion} we know that the map $\mu\mapsto\mu_+$ is measurable and since $\kappa_+$ is the composition of $\kappa'_+$ and of the map $\mu\mapsto\mu_+$, we obtain that $\kappa_+$ is measurable. The same holds true for all the other maps. Since $\kappa_+$ restricted to $\mathcal{M}_S$ is equal to $\kappa'_+$ and the same holds true for all the other maps, we obtain that
	\begin{equation*}
		\mu_+=\alpha_+\sum_{k\leq\kappa_+}\beta_{+,k}\delta_{\sigma_{+,k}}\quad\text{and}\quad \mu_-=\alpha_-+\sum_{k\leq\kappa_-}\beta_{-,k}\delta_{\sigma_{-,k}}
	\end{equation*}
	from which the unique representation (\ref{eq-measure}) follows. 
	
	For the converse, by Lemma 1.6 in \cite{Kallenberg2} we obtain that $\alpha_+$, $\kappa_+$, and $(\beta_{+,k},\sigma_{+,k})$, and $\alpha_-$, $\kappa_-$, and all $(\beta_{-,k},\sigma_{-,k})$, which are here considered as maps from $(A,\mathcal{B})$ to their respective codomain (\textit{e.g.}~$\alpha_+:A\mapsto\mathcal{M}_S$), uniquely determine $\mu_+$ and $\mu_-$ as measurable functions on $A$. Then, by the measurability of addition (Lemma \ref{lem-measurable-vector-space}) and the measurability of the inclusion map (Lemma \ref{lem-inclusion}) we obtain that $\mu$ considered as a map from $(A,\mathcal{B})$ to $(\mathcal{M}_S,\mathcal{B}_{\mathcal{M}_S})$ is measurable. Indeed, let $a$ denote the addition map and let $i$ denote the inclusion map, we have that $\mu(z)=a(i(\mu_+(z)),i(\mu_-(z)))$ for $z\in A$, namely $\mu:(A,\mathcal{B})\to(\mathcal{M}_S\times\mathcal{M}_S,\mathcal{B}_{\mathcal{M}_S}\otimes \mathcal{B}_{\mathcal{M}_S})\to(\mathsf{M}_S\times\mathsf{M}_S,\mathcal{B}_{\mathsf{M}_S}\otimes \mathcal{B}_{\mathsf{M}_S})\to(\mathsf{M}_S,\mathcal{B}_{\mathsf{M}_S})$. Since $\mu_+$, and $\mu_-$ uniquely determine $\mu$ we obtain our result.
\end{proof}

\begin{proof}[Proof of Corollary \ref{co-atomic decomposition}]
	It follows from Lemma \ref{lem-deterministic-representation}.
\end{proof}

\begin{proof}[Proof of Proposition \ref{pro-fdd=d}]
	One direction is trivial. For the other we proceed as follows. Denote by $\Sigma$ the $\sigma$-algebra in $\mathsf{M}_{S}$ generated by all projections $\pi_{A}:\mu\mapsto\mu(A)$, $A\in\mathcal{I}$. Consider the class $\mathcal{C}$ of subsets $A$ of $S$ for which $\pi_{A}$ is $\Sigma$-measurable. Since 
	\begin{align*}
		\pi_{A\cup B}(\mu)&=\mu(A\cup B)=\mu(A)+\mu(B)=\pi_A(\mu)+\pi_B(\mu),\\
		\pi_{C\setminus D}(\mu)&=\mu(C\setminus D)=\mu(C)-\mu(D)=\pi_C(\mu)-\pi_D(\mu),
	\end{align*}
	for every $A,B,C,D\in\mathcal{C}$ with $A,B$ disjoints and $C\supseteq D$ and for every $\mu\in\mathsf{M}_S$, we obtain that $\pi_{A\cup B}\equiv\pi_A+\pi_B$ and $\pi_{C\setminus D}\equiv\pi_C-\pi_D$. Since the sum and the difference of two measurable functions are measurable, we have that $\pi_{A\cup B}$ and $\pi_{C\setminus D}$ are $\Sigma$-measurable and so $\mathcal{C}$ is closed under finite union and proper difference (and so closed under finite intersection and complementation). 
	
	Moreover, let $(A_n)$ to be a sequence of elements in $\mathcal{C}$ such that $A_n\uparrow A$ (or $A_n\downarrow A$) where $A\in\hat{\mathcal{S}}$. The class $\{B\in\mathcal{S}:B\subseteq A\}$ is a $\sigma$-algebra (see also Lemma \ref{215}) and so any $\mu\in \mathsf{M}_{S}$ is a well-defined bounded signed measure when restricted to $A$ (or to a bounded set large enough to include all the sets $(A_n)_{n\geq n'}$, for some $n'$ large enough, when $A_n\downarrow A$). Then, by continuity lemma $\mu(A_n)\to\mu(A)$, for every $\mu\in \mathsf{M}_{S}$, and so we obtain that $\pi_{A_n}\to\pi_A$ pointwise, as $n\to\infty$. Since the limit of a sequence of real-valued measurable functions which converges pointwise has a measurable limit, we obtain that $\pi_A$ is $\Sigma$-measurable. Thus, $\mathcal{C}$ is a (local) monotone class. Since $\mathcal{C}\supset\mathcal{I}$, by Lemma 1.2 point (ii) in \cite{Kallenberg2} it follows that $\mathcal{C}$ includes the $\sigma$-ring generated by $\mathcal{I}$, which means that any bounded set of the $\sigma$-ring generated by $\mathcal{I}$ is also $\Sigma$-measurable. But since the $\sigma$-ring generated by $\mathcal{I}$ is $\mathcal{S}$, we have that any $A\in\hat{\mathcal{S}}$ is $\Sigma$-measurable. Hence, $\Sigma$ contains the $\sigma$-algebra in $\mathsf{M}_{S}$ generated by all projections $\pi_{A}:\mu\mapsto\mu(A)$, $A\in\hat{\mathcal{S}}$. Since $\mathcal{I}\subset\hat{\mathcal{S}}$ we conclude that $\Sigma$ coincides with $\mathcal{B}_{\mathsf{M}_S}$.
	
	Now, consider the class $\mathcal{E}$ of subsets of $\mathsf{M}_S$ given by $\{\mu\in\mathsf{M}_S:\mu(A_1)\in B_1,...,\mu(A_k)\in B_k\}$ where $A_1,...,A_k\in\mathcal{I}$, $B_1,...,B_k\in\mathcal{B}(\mathbb{R})$, and $k\in\mathbb{N}$. We show that $\mathcal{E}$ is a semi-ring. Since 
	\begin{align*}
		&\{\mu\in\mathsf{M}_S:\mu(A_1)\in B_1,...,\mu(A_k)\in B_k\}\\&	=\{\mu\in\mathsf{M}_S:\mu(A_1)\in B_1,...,\mu(A_k)\in B_k,\mu(A_{k+1})\in\mathbb{R},...,\mu(A_{k+m})\in\mathbb{R}\}
	\end{align*}
	
	for every $m\in\mathbb{N}$, we can always assume that $A_1,...,A_k$ are the same for the sets under consideration. First, we show that $\mathcal{E}$ is closed under finite intersections. So, consider two sets $E_1,E_2\in\mathcal{E}$ with say 
	\begin{align*}
		E_1=&\{\mu\in\mathsf{M}_S:\mu(A_1)\in B_1,...,\mu(A_k)\in B_k\},\\
		E_2=&\{\mu\in\mathsf{M}_S:\mu(A_1)\in \tilde{B}_1,...,\mu(A_k)\in \tilde{B}_k\}.
	\end{align*}
	Then, 
	\begin{equation*}
		E_1\cap E_2=\{\mu\in\mathsf{M}_S:\mu(A_1)\in B_1\cap\tilde{B}_1,...,\mu(A_k)\in B_k\cap\tilde{B}_k\}
	\end{equation*}
	and since intersections of Borel sets are Borel sets we have that $B_1\cap\tilde{B}_1,...,B_k\cap\tilde{B}_k$, are Borel sets and so $E_1\cap E_2\in\mathcal{E}$. Now, we need to show that given $E_1,E_2\in\mathcal{E}$ there exist disjoint $F_1,...,F_n\in\mathcal{E}$ such that $E_1\setminus E_2=\bigcup_{i=1}^n F_i$ for some $n\in\mathbb{N}$. Consider $E_1$ and $E_2$ as before. Then, 
	\begin{equation*}
		E_1\setminus E_2=\{\mu\in\mathsf{M}_S:(\mu(A_1),...,\mu(A_k))\in (B_1,...,B_k)\setminus(\tilde{B}_1,...,\tilde{B}_k)\}.
	\end{equation*}
	If we show that $		(B_1,...,B_k)\setminus(\tilde{B}_1,...,\tilde{B}_k)=\bigcup_{i=1}^n C_i$ for some disjoint $k$-fold product of Borel sets $C_1,...,C_n$ and some $n\in\mathbb{N}$, then we would have that
	\begin{equation*}
		E_1\setminus E_2=\bigcup_{i=1}^n\{\mu\in\mathsf{M}_S:(\mu(A_1),...,\mu(A_k))\in C_i\}
	\end{equation*}
	and so we would obtain the result. It is possible to see that 
	\begin{equation*}
		(B_1,...,B_k)\setminus(\tilde{B}_1,...,\tilde{B}_k)=\bigcup(Z_1,...,Z_k),
	\end{equation*}
	where the union runs through all the sets of the form $(Z_1,...,Z_k)$ such that $Z_i\in\{B_i\cap\tilde{B}_i,B_i\setminus\tilde{B}_i\}$, for each $i=1,...,k$, excluding the set $(B_1\cap\tilde{B}_1,...,B_k\cap\tilde{B}_k)$. Thus, it is a union of $2^{k}-1$ sets and the sets are disjoint, this is because for every two sets $(Z_1,...,Z_k)$ and $(Z'_1,...,Z'_k)$ there will be at least one $l\in\{1,...,k\}$ such that $Z_l=B_l\cap\tilde{B}_l$ and $Z'_l=B_l\setminus\tilde{B}_l$.
	
	Thus, $\mathcal{E}$ is a semi-ring and it is possible to see that it generates $\Sigma$. Therefore, by the fact that $(\xi(I_1),...,\xi(I_n))\stackrel{d}{=}(\eta(I_1),...,\eta(I_n))$, for every $I_1,...,I_n\in\mathcal{I}$, by the fact that $\mathcal{E}$ is a semi-ring and it generates $\Sigma$, by Proposition A1.3.I (a) and (b) in \cite{Daley1}, and by the countable additivity of the probability measure on $\mathsf{M}_S$, we obtain that $\xi\stackrel{d}{=}\eta$ on $\Sigma$. But since $\Sigma$ coincides with $\mathcal{B}_{\mathsf{M}_S}$, we obtain the stated result.
\end{proof}

\begin{proof}[Proof of Lemma \ref{lem-kernel}]
	First, observe that in all cases (namely $\textnormal{(i)}$, $\textnormal{(ii)}$, and $\textnormal{(iii)}$) we have that $\xi(\omega,\cdot )$ is a countably additive set function on $\hat{\mathcal{S}}$, for every fixed $\omega\in\Omega$. 
	
	We first show that $\textnormal{(i)}\Leftrightarrow\textnormal{(ii)}$. To show that $\textnormal{(ii)}\Rightarrow\textnormal{(i)}$ notice $\mathcal{B}_{\mathsf{M}_S}$ is generated by the projections map $\pi_{B}:\mu\mapsto\mu(B)$ for all $B\in\hat{\mathcal{S}}$, hence $\pi_B$ is a measurable function from $(\mathsf{M}_S,\mathcal{B}_{\mathsf{M}_S})$ to $(\mathbb{R},\mathcal{B}(\mathbb{R}))$ for all $B\in\hat{\mathcal{S}}$. Since a composition of measurable functions is measurable, we have that $\xi(\cdot,B)=\pi_B\circ\xi$ is $\mathcal{F}$-measurable. Thus, $\textnormal{(ii)}\Rightarrow\textnormal{(i)}$. 
	
	For the converse we need to show that $\xi$ is a measurable function from $(\Omega,\mathcal{F},\mathbb{P})$ to $(\mathsf{M}_S,\mathcal{B}_{\mathsf{M}_S})$, that is $\xi^{-1}(A)\in\mathcal{F}$ for every $A\in \mathcal{B}_{\mathsf{M}_S}$. By Lemma 1.4 in \cite{Kallenberg} it is enough to show that $\xi^{-1}(C)\in\mathcal{F}$ for every $C$ belonging to the set that generates $\mathcal{B}_{\mathsf{M}_S}$. Thus, it is sufficient to show that
	$\{\omega\in\Omega: \xi(\omega)\in\{\mu\in\mathsf{M}_S:\mu(B)\in D\} \}$ belongs to $\mathcal{F}$, where $B\in\hat{\mathcal{S}}$ and $D\in\mathcal{B}(\mathbb{R})$. But this follows from the $\mathcal{F}$-measurability of the pre-kernel $\xi$. Indeed, 
	\begin{align*}
		&\{\omega\in\Omega: \xi(\omega)\in\{\mu\in\mathsf{M}_S:\mu(B)\in D\}\}=	\{\omega\in\Omega: \pi_B(\xi(\omega))\in D\}\\&	=\{\omega\in\Omega: \xi(\omega,B)\in D\}\in\mathcal{F}.
	\end{align*}
	Thus, $\textnormal{(i)}\Rightarrow\textnormal{(ii)}$.
	
	We now show that $\textnormal{(i)}\Leftrightarrow\textnormal{(iii)}$. Since $\textnormal{(i)}$ immediately implies $\textnormal{(iii)}$, it remains to show that $\textnormal{(iii)}\Rightarrow\textnormal{(i)}$. We follow similar arguments as the ones in the proof of Proposition \ref{pro-fdd=d}. Let $\mathcal{C}'$ be the class of subsets of $B\in\hat{\mathcal{S}}$ such that $\xi(\cdot,B)$ is $\mathcal{F}$-measurable, then $\mathcal{C}'$ is a local monotone class and the $\mathcal{F}$-measurability extends to any bounded set of the $\sigma$-ring generated by $\mathcal{I}$ and since the $\sigma$-ring generated by $\mathcal{I}$ is $\mathcal{S}$ we have that $\xi(\cdot,B)$ is $\mathcal{F}$-measurable for every $B\in\hat{\mathcal{S}}$. Thus, we obtain $\textnormal{(i)}$.
\end{proof}

\begin{proof}[Proof of Corollary \ref{co-uncited-4}]
	Since $\xi$ is a random signed measure, by Lemma \ref{lem-Jordan} there exist two random measures $\xi_+$ and $\xi_-$ such that $\xi=\xi_+-\xi_-$ on $\hat{\mathcal{S}}$. Then, we obtain the result by Corollary 2.17 in \cite{Kallenberg2}.
\end{proof}

\begin{proof}[Proof of Lemma \ref{lem-correspondence-pos}]
	Let $B$ be the open interval $(a,b)$ for some $0<a<b<\infty$. Since $\mu(A)<\infty$, there are only finitely many atoms with values in $(a,b)$. Thus, for any dissection system $(A_{nj})$ there is an $n'$ large enough such that these atoms are all in different $A_{nj}$'s for every $n>n'$. Moreover, these atoms have values in $(a+\varepsilon,b-\varepsilon)$, that is in $B_\varepsilon$, for any $0<\varepsilon<\varepsilon^*$, for some $\varepsilon^*>0$.  Thus, the problem is to avoid counting atoms which have values outside $(a,b)$. If an atom has value greater than or equal to $b$ then it will never be counted, thus such atoms are not a problem. For the atoms with values smaller than or equal to $a$ we have the following. 
	
	In this case, as shown in Example \ref{counterexample}, the problem is that there might be a concentration of atoms around certain points whose values when summed are above $a$. Since $\mu(A)<\infty$, there are only finitely many of such concentrations of atoms because there can only be finitely many $A_{nj}$'s such that $\mu(A_{nj})>a$, in particular there can be at most $\frac{\mu(A)}{a}$ many of them. In the worst case, the concentration is around an atom of value $a$. Then for every $\varepsilon>0$ we will have that there is a $\tilde{n}_\varepsilon$ large enough such that for the $A_{nj}$ containing such an atom we will have that $\mu(A_{nj})<a+\varepsilon$. Thus, such atom will not be counted. Since there are finitely many such atoms, there is going to be an $n_\varepsilon$ large enough such that none of these atoms are counted. Therefore, for every $\varepsilon<\varepsilon^*$ and for every $n>n_\varepsilon$, for some $n_\varepsilon\in\mathbb{N}$, we have that $\sum_{i}\delta_{(s_i,\kappa_i)}(A\times B)=\sum_{j}\delta_{\mu(A_{nj})}(B_{\varepsilon})$, from which we obtain the result. The same arguments apply to the case of $B$ being any open set. The arguments for $B$ closed are similar.
\end{proof}

\begin{corollary}\label{co-9.1.VII}
	Equations $(\ref{D1})$,  $(\ref{D2})$ and  $(\ref{D3})$ establish a one-to-one correspondence between purely atomic random measure $\xi$ and marked point process, $N_{\xi}$ say, satisfying the condition $(\ref{D0})$.
\end{corollary}
\begin{proof}
	It follows directly from Lemma \ref{lem-correspondence-pos}.
\end{proof}

\begin{proof}[Proof of Lemma \ref{lem-correspondence}]
	Since $\xi$ is purely atomic, we have that $\xi_+$ and $\xi_-$ are purely atomic. So, by Corollary \ref{co-9.1.VII} we obtain that $\xi_+$ has a one-to-one almost sure correspondence with a marked point process, say $N_{\xi_+}$. Now, by slightly adapting the arguments of Lemma \ref{lem-correspondence-pos} we obtain that there exists a one-to-one correspondence, both ways measurable, between a purely atomic signed measure $-\mu$ where $\mu\in\mathcal{M}_S$ and  a counting measure $N_{-\mu}(A\times B)=\sum_{i}\delta_{(s_i,x_i)}(A\times B)$ on the Borel sets of $S\times (-\infty,0)$ such that for every $A\in\hat{\mathcal{S}}$
	\begin{equation*}
		\int_{A\times (-\infty,0)}x N_{-\mu}(ds\times dx)=\sum_{i:s_i\in A}x_i>-\infty,
	\end{equation*}
	and the correspondence for every $A\in\hat{\mathcal{S}}$ and $B\in \mathcal{B} ((-\infty,0))$ with $0\notin\bar{B}$ is given by 
	\begin{align*}
		-\mu(A)&=	\int_{A\times (-\infty,0)} xN_{-\mu}(ds\times dx)\quad\textnormal{and}\\
		N_{-\mu}(A\times B)&=\lim\limits_{\varepsilon\to 0}\lim\limits_{n\to\infty}\sum_{j}\delta_{-\mu(A_{nj})}(B_{\varepsilon}),\quad\text{for $B$ open,}\\
		N_{-\mu}(A\times B)&=\lim\limits_{\varepsilon\to 0}\lim\limits_{n\to\infty}\sum_{j}\delta_{-\mu(A_{nj})}(B^{\varepsilon}),\quad\text{for $B$ closed,}
	\end{align*}
	where $(A_{nj})$ is a dissection system of measurable subsets of $A$. Thus, we obtain that $-\xi_-$ has a one-to-one a.s.~correspondence with a marked point process with only negative marks, say $N_{-\xi_-}$. Alternatively, by Corollary \ref{co-9.1.VII} we have a one-to-one correspondence between $\xi_-$ and a marked point process, say $\tilde{N}_{\xi_-}$, and by setting $N_{\xi_-(\omega)}(A\times B):=\tilde{N}_{\xi_-(\omega)}(A\times -B)$  for every $\omega\in\Omega$, $A\in\hat{\mathcal{S}}$ and $B\in \mathcal{B} ((-\infty,0))$, we get
	\begin{equation*}
		-\int_{A\times(0,\infty)} s\tilde{N}_{\xi_-(\omega)}(ds\times dx)=-\sum_{i:s_i\in A}x_i		=\sum_{i:s_i\in A}-x_i=\int_{A\times (-\infty,0)} sN_{-\xi_-(\omega)}(ds\times dx).
	\end{equation*}
	Now, extend $N_{\xi_+}$ to $S\times (\mathbb{R}\setminus\{0\})$ by setting $N_{\xi_+}\equiv0$ on $S\times(-\infty,0)$ (and call also this extension $N_{\xi_+}$), and similarly extend $N_{-\xi_-}$ to $S\times (\mathbb{R}\setminus\{0\})$ by setting $N_{\xi_-}\equiv0$ on $S\times(0,\infty)$ (and call this extension $N_{-\xi_-}$). Now, let $N_\xi\stackrel{a.s.}{=}N_{\xi_+}+N_{-\xi_-}$. Notice that $N_\xi$ is a well-defined random measure since is the sum of two random measures and that there is a one-to-one correspondence between $(N_{\xi_+},N_{-\xi_-})$ and $N_{\xi}$ because $N_{\xi_+}(\cdot)=N_{\xi}(S\times(0,\infty)\cap \cdot)$ and $N_{-\xi_-}(\cdot)=N_{\xi}(S\times(-\infty,0)\cap \cdot)$ by definition. 
	
	Further, since there is a one-to-one correspondence between $\xi$ and $(\xi_+,\xi_-)$ and between $(N_{\xi_+},N_{-\xi_-})$ and $N_{\xi}$, and a one-to-one a.s.~correspondence between $(\xi_+,\xi_-)$ and $(N_{\xi_+},N_{-\xi_-})$, we obtain a one-to-one a.s.~correspondence between $\xi$ and $N_{\xi}$.
	
	Moreover, since $\xi=\xi_+-\xi_-$, we have that equations (\ref{eq-0}), (\ref{eq-1}) and (\ref{eq-2}) follow if and only if $N_\xi\stackrel{a.s.}{=}N_{\xi_+}+N_{\xi_-}$. Indeed, we have the following. First, almost surely we have
	\begin{equation*}
		\int_{A\times \mathbb{R}\setminus\{0\}} sN_{\xi}(ds\times dx)= \int_{A\times \mathbb{R}\setminus\{0\}} sN_{\xi_+}(ds\times dx)+	\int_{A\times \mathbb{R}\setminus\{0\}} sN_{-\xi_-}(ds\times dx)
	\end{equation*}
	\begin{equation*}
		=\xi_+(A)-	\xi_-(A)= \xi(A)
	\end{equation*}
	and
	\begin{equation*}
		\int_{A\times \mathbb{R}\setminus\{0\}} |s|N_{\xi}(ds\times dx)= |\xi|(A)<\infty.
	\end{equation*}
	Second, let $B$ be the open set $(a,b)$, for some $0<a<b<\infty$. Fix $\omega\in\Omega$. We will show that
	\begin{equation}\label{delta}
		\lim\limits_{\varepsilon\to0}\lim\limits_{n\to\infty}\sum_{j}\delta_{\xi_+(\omega, A_{nj})}(B_\varepsilon)+	\lim\limits_{\varepsilon\to0}\lim\limits_{n\to\infty}\sum_{j}\delta_{-\xi_-(\omega,A_{nj})}(B_{\varepsilon})				=	\lim\limits_{\varepsilon\to0}\lim\limits_{n\to\infty}\sum_{j}\delta_{\xi(\omega,A_{nj})}(B_{\varepsilon}).
	\end{equation}
	Since $|\xi|(\omega,A)<\infty$, $|\xi|(\omega,\cdot)$ and so $\xi(\omega,\cdot)$ have only finitely many atoms with values in $(-\infty,b]\cup[b,\infty])$. Further, by definition $\xi_+(\omega,\cdot)$ and $\xi_-(\omega,\cdot)$ have atoms at different locations. Thus, there is an $n'$ large enough such that each of these atoms belongs to a different $A_{nj}$ for every $n>n'$. The following arguments are similar to the ones used in the proof of Lemma \ref{lem-correspondence-pos}; where now we use that $|\xi|(\omega,A)<\infty$ instead of $\mu(A)<\infty$. 
	
	The atoms of $\xi(\omega,\cdot)$ of values in $(a,b)$ belong to $(a+\varepsilon,b-\varepsilon)$, that is in $B_\varepsilon$, for any $0<\varepsilon<\varepsilon^*$, for some $\varepsilon^*>0$. Thus, for every $n$ large enough these atoms are always counted, namely there is an $n''$ such that the $A_{nj}$'s containing each of these atoms are such that $\delta_{\xi(\omega,A_{nj})}(B_{\varepsilon})=1$ for every $n>n''$. For the atoms with values greater than or equal to $b$ we have the following. Consider an atom with value greater than or equal to $b$. Then there might be a concentration of atoms with negative values around it. Notice that there can be at most $\frac{|\xi|(\omega,A)}{b}$ many such atoms. Thus, for any $\varepsilon$ there is an $n$ large enough such that these atoms are never counted. For the atoms with values in $[-a,a]$ the same arguments put forward in the proof of Lemma \ref{lem-correspondence-pos} for the case of atoms with values smaller than or equal to $a$ hold here (again using that $|\xi|(\omega,A)<\infty$). For the atoms with values in $(-\infty,-a)$ we have that they are finitely many and so $A_{nj}$'s containing each of these atoms are such that $\delta_{\xi(\omega,A_{nj})}(B_{\varepsilon})=0$ for every $n$ large enough, hence they will never be counted. 
	
	Therefore, for any $\varepsilon<\varepsilon^*$ and for every $n>n_\varepsilon$, for some $n_\varepsilon\in\mathbb{N}$, we have that 
	\begin{equation*}
		\sum_{i}\delta_{(s_i,\kappa_i)}(A\times B)=\sum_{j}\delta_{\xi(\omega,A_{nj})}(B_{\varepsilon})
	\end{equation*}
	where $\{(s_i,\kappa_i)_{i\in\mathbb{N}}\}$ are the atoms of $\xi(\omega,\cdot)$ with their respective values, thus $\kappa_i=\xi(\omega,\{s_i\})$. Using that $\xi_+(\omega,\cdot)$ and $\xi_-(\omega,\cdot)$ have atoms at different locations then by denoting $\{(s_{+,i},\kappa_{+,i})_{i\in\mathbb{N}}\}$ and $\{(s_{-,i},\kappa_{-,i})_{i\in\mathbb{N}}\}$ the atoms with respective weights of $\xi_+(\omega,\cdot)$ and $\xi_-(\omega,\cdot)$, respectively, we obtain that
	\begin{equation*}
		\sum_{i}\delta_{(s_{+,i},\kappa_{+,i})}(A\times B)+\sum_{j}\delta_{(s_{-,j},\kappa_{-,j})}(A\times B)=\sum_{i}\delta_{(s_i,\kappa_i)}(A\times B)=\sum_{j}\delta_{\xi(\omega,A_{nj})}(B_{\varepsilon})
	\end{equation*}
	from which we obtain (\ref{delta}). The same arguments apply to the case of $B$ being any open set and similar arguments apply for $B$ being closed. Thus, we have almost surely 
	\begin{align*}
		N_{\xi}(A\times B)=&N_{\xi_+}(A\times B)+N_{-\xi_-}(A\times B)\\
		=&\lim\limits_{\varepsilon\to 0}\lim\limits_{n\to\infty}\sum_{j}\delta_{\xi_+(A_{nj})}(B_\varepsilon)+	\lim\limits_{\varepsilon\to0}\lim\limits_{n\to\infty}\sum_{j}\delta_{-\xi_-(A_{nj})}(B_\varepsilon)\\
		=&\lim\limits_{n\to\infty}\sum_{j}\delta_{\xi(A_{nj})}(B_\varepsilon),\quad\textnormal{for $B$ open and}\\
		N_{\xi}(A\times B)=&N_{\xi_+}(A\times B)+N_{-\xi_-}(A\times B)\\
		=&\lim\limits_{\varepsilon\to 0}\lim\limits_{n\to\infty}\sum_{j}\delta_{\xi_+(A_{nj})}(B^\varepsilon)+\lim\limits_{\varepsilon\to 0}\lim\limits_{n\to\infty}\sum_{j}\delta_{-\xi_-(A_{nj})}(B^\varepsilon)\\
		=&\lim\limits_{n\to\infty}\sum_{j}\delta_{\xi(A_{nj})}(B^\varepsilon),\quad\textnormal{for $B$ closed.}
	\end{align*}
\end{proof}

\begin{proof}[Proof of Theorem \ref{thm-representation}]
	First, notice that if (\ref{eq-representation}) holds then $\xi$ has independent increments. For the other direction we proceed as follows.
	
	By Corollary \ref{co-atom} $\xi$ has at most countably many fixed points. By independence of the increments these fixed points constitute the first addendum in (\ref{eq-representation}). Thus, for the rest of this proof we let $\xi$ have no fixed atoms. 
	
	By Theorem 2.2 in \cite{PrekopaIII} we have that $\xi(A)$ is infinitely divisible for every $A\in\hat{\mathcal{S}}$. In particular, while Theorem 2.2 in \cite{PrekopaIII} applies to random measures on $\sigma$-rings, one can restrict $\xi$ to $\{B\in\hat{\mathcal{S}}:B\subset A\}$ which is a $\sigma$-algebra (thanks to Lemma \ref{215}) and so a $\sigma$-ring. Then, $\xi$ is an infinitely divisible independently scattered random measure according to \cite{RajRos} and by Lemmas 2.1 and 2.3 in \cite{RajRos} we have that the characteristic function of $\xi(A)$, for every $A\in\hat{\mathcal{S}}$ and $t\in\mathbb{R}$, is given by
	\begin{equation*}
		\exp\left(it\bigg(\nu_0(A)-\int_{|x|\leq 1}it\tau(x) F(A\times dx)\bigg)-\frac{1}{2}t^2\nu_1(A)+\int_{\mathbb{R}}e^{itx}-1 F(A\times dx)\right)
	\end{equation*}
	where 
	\begin{equation*}
		\tau(x):=\begin{cases}
			x\,\,&\textnormal{if }\,\,\,|x|\leq 1,\\ \frac{x}{|x|} \,\,&\textnormal{if }\,\,\,|x|> 1,
		\end{cases}
	\end{equation*}
	$\nu_0$ is the difference of two measures in $\mathcal{M}_S$, $\nu_1\in\mathcal{M}_S$, and by Lemma 2.3 in \cite{RajRos} $F$ is measure on $\mathcal{S}\otimes\mathcal{B} (\mathbb{R})$ such that $F(A\times \cdot)$ is a L\'{e}vy measure, that is $\int_{\mathbb{R}}1\wedge x^2 F(A\times dx)<\infty$ and $F(A\times \{ 0\})=0$. By Theorem 2.3 in \cite{PrekopaIII} $\nu_0$, $\nu_1$, and $A\mapsto F(A\times B)$, for every $B\in \mathcal{B} (\mathbb{R})$ such that $0\notin\bar{B}$, are atomless. Since $F(A\times \{ 0\})=0$, we can rewrite the characteristic function of $\xi(A)$ as
	\begin{equation*}
		\exp\left(it\bigg(\nu_0(A)-\int_{|x|\leq 1}it\tau(x) F(A\times dx)\bigg)-\frac{1}{2}t^2\nu_1(A)+\int_{\mathbb{R}\setminus\{0\}}e^{itx}-1 F(A\times dx)\right).
	\end{equation*}
	In particular, by simple arguments as the ones put forward in \cite{HELLMUND}, $\xi$ has no Gaussian component and the number of small jumps cannot be too explosive, namely we have $\nu_1\equiv 0$ and
	\begin{equation*}
		\int_{\mathbb{R}\setminus\{0\}}1\wedge|x| F(A\times dx)<\infty,\quad\textnormal{for every $A\in\hat{\mathcal{S}}$.}
	\end{equation*}
	
	Moreover, by Campbell's theorem and by adapting the arguments in Kingman's book \cite{King} there is a Poisson process $\Psi$ on $S\times(\mathbb{R}\setminus\{0\})$ such that
	\begin{equation*}
		\mathbb{E}\bigg[\exp\bigg(it\int_{\mathbb{R}\setminus\{0\}}x\Psi(A\times dx)\bigg)\bigg]=	\exp\left(\int_{\mathbb{R}\setminus\{0\}}e^{itx}-1 F(A\times dx)\right),
	\end{equation*}
	for every $t\in\mathbb{R}$ and $A\in\hat{\mathcal{S}}$. Then, a potential candidate for the deterministic component $\alpha$ is then given by 
	\begin{equation*}
		\nu_0(A)-\int_{|x|\leq 1}itx F(A\times dx)  ,\quad\textnormal{for every $A\in\hat{\mathcal{S}}$.}
	\end{equation*}
	We have now described the distribution of $\xi$, in the following we show the almost sure representation. 
	
	From Lemma \ref{lem-correspondence} we can identify the component with random atoms as a signed marked signed point process on $\mathcal{S}$, say $N_\xi$. We now prove that $N_\xi$ is a Poisson random measure on $S\times(\mathbb{R}\setminus\{0\})$ by generalizing the arguments in the proof of Theorem 10.1.III in \cite{Daley}. Since by construction $N_\xi$ is a simple point process, by Theorem 3.17 in \cite{Kallenberg2} it is enough to show that $N_\xi$ has independent increments. 
	
	So consider two sets $D_1=A_1\times (a_1,b_1)$ and $D_2=A_2\times (a_1,b_1)$ for some sets $A_1,A_2\in\hat{\mathcal{S}}$ with $A_1\cap A_2=\emptyset$ and $a_1<b_1$ with $a_1,b_1\in (-\infty,0)$ or $a_1,b_1\in (0,\infty)$. Then, by the independence of the increments of $\xi$ we have that $N_\xi(D_1)$ is independent of $N_\xi(D_2)$. Consider the case $D_1=A_1\times (a_1,b_1)$ and $D_2=A_1\times (a_2,b_2)$ for some $a_2<b_2$ with $(a_2,b_2)\cap (a_1,b_1)=\emptyset$ and $a_2,b_2\in (-\infty,0)$ or $a_2,b_2\in (0,\infty)$. Let $(A_{nj})$ be a dissection system of measurable subsets of $A$, let $\varepsilon>0$ and let 
	\begin{equation*}
		Z_{\varepsilon, nj}:=\begin{cases}
			1\,&\textnormal{if $\xi(A_{nj})\in(a_1+\varepsilon,b_1-\varepsilon)$},\\0 \,\,&\textnormal{otherwise,}
		\end{cases}
		\quad
		Y_{\varepsilon, nj}:=\begin{cases}
			1\,&\textnormal{if $\xi(A_{nj})\in(a_2+\varepsilon,b_2-\varepsilon)$},\\0 \,\,&\textnormal{otherwise,}
		\end{cases}
	\end{equation*}
	and let $p_{\varepsilon, nj}=\mathbb{P}(Z_{\varepsilon,nj}=1)$ and $q_{\varepsilon, nj}=\mathbb{P}(Y_{\varepsilon,nj}=1)$. Then, by (\ref{eq-2}) we have $N_\xi(D_1)\stackrel{a.s.}{=}\lim\limits_{\varepsilon\to 0}\lim\limits_{n\to\infty}\sum_{j=1}^{k_n}Z_{\varepsilon,nj}$ and $N_\xi(D_2)\stackrel{a.s.}{=}\lim\limits_{\varepsilon\to 0}\lim\limits_{n\to\infty}\sum_{j=1}^{k_n}Y_{\varepsilon,nj}$. The independence of the increments property of $\xi$ and the dominated convergence theorem imply that for the joint probability generating function
	\begin{align}
		&\mathbb{E}\big[t_1^{N_\xi(D_1)}t_2^{N_\xi(D_2)}\big]=\lim\limits_{\varepsilon\to 0}\lim\limits_{n\to\infty}\prod_{j=1}^{k_n}\mathbb{E}\big[t_1^{Z_{\varepsilon,nj}}t_2^{Y_{\varepsilon,nj}}\big]\nonumber\\&	=\lim\limits_{\varepsilon\to 0}\lim\limits_{n\to\infty}\prod_{j=1}^{k_n}[1-p_{\varepsilon,nj}(1-t_1)-q_{\varepsilon,nj}(1-t_2)]\label{eq-t_1}
	\end{align}
	and
	\begin{align}
		&\mathbb{E}\big[t_1^{N_\xi(D_1)}\big]\mathbb{E}\big[t_2^{N_\xi(D_2)}\big]=\lim\limits_{\varepsilon\to 0}\lim\limits_{n\to\infty}\prod_{j=1}^{k_n}[1-p_{\varepsilon,nj}(1-t_1)][1-q_{\varepsilon,nj}(1-t_2)]\nonumber\\&	=\lim\limits_{\varepsilon\to 0}\lim\limits_{n\to\infty}\prod_{j=1}^{k_n}[1-p_{\varepsilon,nj}(1-t_1)-q_{\varepsilon,nj}(1-t_2)+p_{\varepsilon,nj}q_{\varepsilon,nj}(1-t_1)(1-t_2)]\label{eq-t_2}
	\end{align}
	where $t_1,t_2\in[0,1]$. If (\ref{eq-t_1}) and (\ref{eq-t_2}) are equal for all $t_1,t_2\in[0,1]$, then $N_\xi(D_1)$ and $N_\xi(D_2)$ are independent. Let $r_{\varepsilon,nj}=p_{\varepsilon,nj}(1-t_1)+q_{\varepsilon,nj}(1-t_2)$. Using the inequalities $x\leq-\log(1-x)\leq x/(1-x)$ for $0\leq x<1$, we obtain that for every $t_1,t_2\in[0,1]$
	\begin{equation}\label{10.1.9}
		0\leq -\sum_{j=1}^{k_n}\log(1-r_{\varepsilon,nj})-\sum_{j=1}^{k_n}r_{\varepsilon,nj}\leq R_{\varepsilon,n}:=\sum_{j=1}^{k_n}\frac{r_{\varepsilon,nj}^2}{1-r_{\varepsilon,nj}}\leq \frac{\max_j r_{\varepsilon,nj}}{1-\max_j r_{\varepsilon,nj}}\sum_{j=1}^{k_n}r_{\varepsilon,nj}.
	\end{equation}
	By the first inequality of (\ref{10.1.9}) we obtain that $		\sum_{j=1}^{k_n}r_{\varepsilon,nj}\leq -\sum_{j=1}^{k_n}\log(1-r_{\varepsilon,nj})$ and by taking the logarithm in (\ref{eq-t_1}) we obtain that
	\begin{equation*}
		\lim\limits_{\varepsilon\to 0}\lim\limits_{n\to\infty}-\sum_{j=1}^{k_n}\log(1-r_{\varepsilon,nj})=-\log 	\mathbb{E}\big[t_1^{N_\xi(D_1)}t_2^{N_\xi(D_2)}\big]
	\end{equation*}
	which is finite. Further, by Lemma 9.3.II in \cite{Daley} applied to $|\xi|$ we obtain that $\max_j r_{\varepsilon,nj}\to 0$ as $n\to\infty$. Therefore, $\lim\limits_{\varepsilon\to 0}\lim\limits_{n\to\infty}R_{\varepsilon,n}=0$. Now, we focus on (\ref{eq-t_2}). We can apply the same arguments used for $r_{\varepsilon,nj}$ to 
	\begin{equation*}
		\tilde{r}_{\varepsilon,nj}=p_{\varepsilon,nj}(1-t_1)+q_{\varepsilon,nj}(1-t_2)-p_{\varepsilon,nj}q_{\varepsilon,nj}(1-t_1)(1-t_2),
	\end{equation*}
	and obtain that
	\begin{equation*}
		\mathbb{E}\big[t_1^{N_\xi(D_1)}\big]\mathbb{E}\big[t_2^{N_\xi(D_2)}\big]=\lim\limits_{\varepsilon\to 0}\lim\limits_{n\to\infty}\sum_{j=1}^{k_n}p_{\varepsilon,nj}(1-t_1)+q_{\varepsilon,nj}(1-t_2)-p_{\varepsilon,nj}q_{\varepsilon,nj}(1-t_1)(1-t_2).
	\end{equation*}
	Since
	\begin{equation*}
		0\leq\lim\limits_{n\to\infty}\sum_{j=1}^{k_n}p_{\varepsilon,nj}q_{\varepsilon,nj}(1-t_1)(1-t_2)\leq  \max_i r_{\varepsilon,ni}\sum_{j=1}^{k_n}r_{\varepsilon,nj},
	\end{equation*}
	we obtain that $\lim\limits_{\varepsilon\to 0}\lim\limits_{n\to\infty}\sum_{j=1}^{k_n}p_{\varepsilon,nj}q_{\varepsilon,nj}(1-t_1)(1-t_2)=0$. Finally, since 
	\begin{align*}
		\mathbb{E}\big[t_1^{N_\xi(D_1)}t_2^{N_\xi(D_2)}\big]=\lim\limits_{\varepsilon\to 0}\lim\limits_{n\to\infty}\sum_{j=1}^{k_n}p_{\varepsilon,nj}(1-t_1)+q_{\varepsilon,nj}(1-t_2),
	\end{align*}
	we obtain that 
	\begin{align*}
		&\mathbb{E}\big[t_1^{N_\xi(D_1)}\big]\mathbb{E}\big[t_2^{N_\xi(D_2)}\big]\\&	=\lim\limits_{\varepsilon\to 0}\lim\limits_{n\to\infty}\sum_{j=1}^{k_n}p_{\varepsilon,nj}(1-t_1)+q_{\varepsilon,nj}(1-t_2)-p_{\varepsilon,nj}q_{\varepsilon,nj}(1-t_1)(1-t_2)\\&=\lim\limits_{\varepsilon\to 0}\lim\limits_{n\to\infty}\sum_{j=1}^{k_n}p_{\varepsilon,nj}(1-t_1)+q_{\varepsilon,nj}(1-t_2)-\lim\limits_{\varepsilon\to 0}\lim\limits_{n\to\infty}\sum_{i=1}^{k_n}p_{\varepsilon,ni}q_{\varepsilon,ni}(1-t_1)(1-t_2)\\&			=	\mathbb{E}\big[t_1^{N_\xi(D_1)}t_2^{N_\xi(D_2)}\big],
	\end{align*}
	where we used that for any sequences of functions $f_n$ and $g_n$ for which the iterated limits $\varepsilon\to 0,n\to\infty$ exist we have $\lim\limits_{\varepsilon\to 0}\lim\limits_{n\to\infty}f_n(\varepsilon)+g_n(\varepsilon)=\lim\limits_{\varepsilon\to 0}\lim\limits_{n\to\infty}f_n(\varepsilon)+\lim\limits_{\varepsilon\to 0}\lim\limits_{n\to\infty}g_n(\varepsilon)$.
	
	The same arguments apply for the case $D_1=A_1\times B_1$ and $D_2=A_1\times B_2$ for $B_1,B_2\subset\mathbb{R}\setminus\{0\}$ such that $0\notin \bar{B}_1$, $0\notin \bar{B}_2$, $B_1\cap B_2=\emptyset$, where each of $B_1$ and $B_2$ is either an open or a closed interval. Consider the intervals $C_1=[a_1,b_1)$, $C_{1,1}=[a_1,d_1]$, and $C_{1,2}=(d_1,b_1)$ and $C_2=[a_2,b_2)$, $C_{2,1}=[a_2,d_2]$, and $C_{2,2}=(d_2,b_2)$ for some $d_1\in (a_1,b_1)$ and $d_2\in (a_2,b_2)$. Since $N_\xi(C_1)\stackrel{a.s.}{=}N_\xi(C_{1,1})+N_\xi(C_{1,2})$ and $N_\xi(C_2)\stackrel{a.s.}{=}N_\xi(C_{1,1})+N_\xi(C_{2,2})$, by similar arguments we obtain the result for this case as well. 
	
	By induction we obtain the mutual independence of the collection of random variables $\{N_\xi(D_j):j=1,...,n \}$ where $D_j$ are rectangular and disjoint, as considered. Therefore, by Proposition 9.2.III in \cite{Daley} we obtain that $N_\xi$ has independent increments. Further, by letting 
	\begin{equation*}
		\alpha_+(B):=\xi_+(B)-\int_{0}^{\infty}x N_{\xi}(B\times dx)\quad\textnormal{and}\quad\alpha_-(B):=\xi_-(B)-\int_{-\infty}^{0}x N_{\xi}(B\times dx),
	\end{equation*}
	we have that $\alpha_+$ and $\alpha_-$ are diffuse random measures (namely without fixed and non-fixed atoms) with independent increments. Hence, by Theorem 3.17 in \cite{Kallenberg2} they are deterministic. In particular, they are atomless measures belonging to $\mathcal{M}_S$. Then, $\alpha$ in the statement is uniquely determined by the difference of $\alpha_+$ and $\alpha_-$.
	
	Finally, by uniqueness of the characteristic function (Lemma 2.1 in \cite{RajRos}) we can identify $\Psi$ with $N_{\xi}$ and obtain the stated properties of its intensity measure. The uniqueness of the representation (\ref{eq-representation}) follows by the uniqueness of $N_\xi$ (by Lemma \ref{lem-correspondence}), which also makes $\alpha_+$ and $\alpha_-$ (and so $\alpha$) uniquely defined, and by the uniqueness of the fixed atomic component by noticing that $X_i=\xi(\{s_i\})$ for any fixed point $s_i$.
\end{proof}

\begin{proof}[Proof of Corollary \ref{co-two-CRSM}]
	It follows from Theorem \ref{thm-representation}.
\end{proof}

\begin{proof}[Proof of Corollary \ref{co-two-CRSM-2}]
	The independence of $\Psi_1$ and $\Psi_2$ follows by Lemma \ref{lem-correspondence-pos}, in particular by the fact that they are measurable functions of $\eta_1$ and $\eta_2$ respectively. Moreover, the fact that $\eta_1$ and $\eta_2$ are the Jordan decomposition of $\xi$ comes from the fact that the random measures $\int_{0}^{\infty}x\Psi_1(\cdot\times dx)$ and $\int_{0}^{\infty}x\Psi_2(\cdot\times dx)$ are independent and their atoms are almost surely distinct because by Theorem 3.1.9 in \cite{Kallenberg2} (see also Theorem \ref{thm-representation}) we have $\mathbb{E}[\Psi_1(\{s\}\times (0,\infty))]=\mathbb{E}[\Psi_2(\{s\}\times (0,\infty))]=0$.
\end{proof}

\begin{proof}[Proof of Proposition \ref{pro-marked-signed}]
	We have that $\mathbb{E}[|\xi|(\{s\}\times T)]>0$ for at most countably many points in $S$, say $s_1,s_2,...\in S$, thus by independence we can focus on $\xi'=\xi-\sum_{k}(\delta_{s_k}\otimes\delta_{\tau_k})$. Notice that $\xi'(\cdot\times T)$ is atomless and has independent increments on $S$, however it is not necessarily locally finite thus we cannot directly apply Corollary \ref{co-two-CRSM}.  So, fix any $B\in \hat{\mathcal{S}}\otimes\hat{\mathcal{T}}$. Observe that $\xi'((\cdot\times T)\cap B)$ is an atomless almost surely finite signed simple point process on $S$ with independent increments. Thus, by Corollary  \ref{co-two-CRSM} its Jordan decomposition $\sup_{E\in \hat{\mathcal{S}}}\xi'(((\cdot\cap E)\times T)\cap B)$ and $\sup_{E\in \hat{\mathcal{S}}}-\xi'(((\cdot\cap E)\times T)\cap B)$ are independent atomless almost surely finite simple point processes with independent increments. Denote them by $\xi'_{+,B}$ and $\xi'_{-,B}$. 
	
	Since $\mathbb{E}[\xi'_{+,B}(\{s\})]\equiv0$ and $\mathbb{E}[\xi'_{-,B}(\{s\})]\equiv0$, by Theorem 3.17 (i) in \cite{Kallenberg2} they are Poisson. In particular, $\xi'_{+,B}(S)$ and $\xi'_{-,B}(S)$ are Poisson random variables. Since $B$ was arbitrary, by Theorem 2.15 in \cite{Kallenberg2} we know that there exists two almost surely unique random measures $\eta_+$ and $\eta_-$ on $S\times T$ such that 
	\begin{equation*}
		\eta_+(B)=\sup_{E\in \hat{\mathcal{S}}}\xi'(((S\cap E)\times T)\cap B)\quad\textnormal{and}\quad
		\eta_-(B)=\sup_{E\in \hat{\mathcal{S}}}-\xi'(((S\cap E)\times T)\cap B),
	\end{equation*}
	for every $B\in \hat{\mathcal{S}}\otimes\hat{\mathcal{T}}$ almost surely. Finally, by Corollary 3.9 in \cite{Kallenberg2} we obtain that $\eta_+$ and $\eta_-$ are independent Poisson random measures, and consequently they are also the Jordan decomposition of $\eta_+-\eta_-$.
\end{proof}

\begin{proof}[Proof of Theorem \ref{thm-ultra-representation}]
	We have that $\mathbb{E}[|\xi|(\{s\}\times T)]>0$ for at most countably many points in $S$ and we may separate the corresponding sum $\sum_{k}(\delta_{s_k}\otimes\beta_k)$, for some distinct points $s_1,s_2,...\in S$ and associated random signed measures $\beta_k$ on $\hat{\mathcal{T}}$, where the latter are mutually independent and independent of the remaining process $\xi'$. 
	
	We now show that the Jordan decomposition of $\xi'$, namely $\xi'_+$ and $\xi'_-$, have independent $S$-increments. Consider any disjoints sets $A,B\in\hat{\mathcal{S}}$ and any sets $C,D\in\hat{\mathcal{T}}$. For any $x,y\in\mathbb{R}$, we have
	\begin{align*}
		&	\mathbb{P}\left(\xi'_+(A\times C)<x,\xi'_+(B\times D)<y\right)\\&	=	\mathbb{P}\bigg(\bigcap_{E_1\times E_2\in\mathcal{R},E_1\subset A, E_2\subset C}\{\xi'(E_1\times E_2)<x\}\cap\bigcap_{F_1\times F_2\in\mathcal{R},F_1\subset B,F_2\subset D}\{\xi'(F_1\times F_2)<y\}\bigg)\\& 	=	\mathbb{P}\bigg(\bigcap_{E_1\times E_2\in\mathcal{R},E_1\subset A, E_2\subset C}\{\xi'(E_1\times E_2)<x\}\bigg)	\mathbb{P}\bigg(\bigcap_{F_1\times F_2\in\mathcal{R},F_1\subset B,F_2\subset D}\{\xi'(F_1\times F_2)<y\}\bigg),
	\end{align*}
	where $\mathcal{R}$ is the countable generating ring of $\mathcal{S}\otimes\mathcal{T}$, which is given by the product of the generating rings of $\mathcal{S}$ and of $\mathcal{T}$. The last equality follows from the fact that \begin{equation*}
		\bigcap_{E_1\times E_2\in\mathcal{R},E_1\subset A, E_2\subset C}\{\xi'(E_1\times E_2)<x\}
	\end{equation*}
	is an element of the $\sigma$-algebra 
	\begin{equation*}
		\sigma\left(\xi'(E_1\times E_2):E_1\times E_2\in\mathcal{R}, E_1\subset A, E_2\subset C\right)
	\end{equation*}
	and this $\sigma$-algebra is independent of  the $\sigma$-algebra 
	\begin{equation*}
		\sigma\left(\xi'(F_1\times F_2):F_1\times F_2\in\mathcal{R}, F_1\subset B, F_2\subset D\right),
	\end{equation*}
	because the random variables $\xi'(E_1\times E_2)$'s and $\xi'(F_1\times F_2)$'s are mutually independent.
	
	Thus, $\xi'_+$ has independent $S$-increments and the same holds for $\xi'_-$. Then, by Theorem 3.19 in \cite{Kallenberg2} we obtain that almost surely
	\begin{equation*}
		\xi'_+=\alpha_++\int\int (\delta_s\otimes\mu)\gamma_+(ds\,d\mu)\quad\textnormal{and}\quad \xi'_-=\alpha_-+\int\int (\delta_s\otimes\mu)\gamma_-(ds\,d\mu),
	\end{equation*}
	for some atomless $\alpha_+,\alpha_-\in\mathcal{M}_{S\times T}$ and some $\gamma_+$ and $\gamma_-$ Poisson processes on $S\times(\mathcal{M}_T\setminus\{0\})$. By the same arguments used to show that $\xi'_+$ has independent $S$-increments we obtain the independence between $\xi'_+(G_1\times G_2)$ and $\xi'_-(H_1\times H_2)$ for any $G_1\times G_2\in\hat{\mathcal{S}}\otimes\hat{\mathcal{T}}$ and $H_1\times H_2\in\hat{\mathcal{S}}\otimes\hat{\mathcal{T}}$ with $G_1$ and $H_1$ disjoint, and in particular we obtain that  $\xi'_+((G_1\times T)\cap \cdot)$ and $\xi'_-((H_1\times T)\cap\cdot)$ are independent random measures.
	
	The atomic part of $\xi'_+$ is given by $\sum_{k}\delta_{\sigma^{+}_{k}}\otimes\psi^+_k$ for some a.s.~distinct random elements $\sigma^+_1,\sigma^+_2,...$ in $S$ and associated random measures $\psi^{+}_k = \xi'_+(\{\sigma^{+}_{k}\}\times\cdot)$ on $T$. This atomic component is encoded in $\gamma_+$, in particular $\gamma_+=\sum_{k}\delta_{\sigma^{+}_{k}}\otimes\delta_{\psi^+_{k}}$, and similarly for $\gamma_-$. We will now show that the $(\mathcal{M}_T\setminus\{0\})$-marked signed point process $\gamma_+-\gamma_-$ has independent $S$-increments. By Theorem 3.19 and its proof in \cite{Kallenberg2}, we know that $\gamma_+$ depends measurably on $\xi'_+$, (and $\gamma_-$ depends measurably on $\xi'_-$). So, consider any sets $A,B\in\hat{\mathcal{S}}$ disjoint and any sets $C=\{\mu\in\mathcal{M}_T:\mu(E)\geq a\}$ and $D=\{\mu\in\mathcal{M}_T:\mu(F)\geq b\}$, for some $E,F\in\hat{\mathcal{T}}$ and $a,b>0$. In the following we use the symbol $\pm$ to indicate both $+$ and $-$; for example if we have $\gamma_\pm$ in an equation then we can substitute it with either $\gamma_+$ or $\gamma_-$. Then, we have 
	\begin{align*}
		&\mathbb{P}\Big(\gamma_\pm(A\times C)<x,\gamma_\pm(B\times D)<y\Big)\\&	=\mathbb{P}\left(\sum_{k}\delta_{\sigma^{\pm}_{k}}\otimes\delta_{\psi^\pm_{k}}(A\times C)<x,\sum_{k}\delta_{\sigma^{\pm}_{k}}\otimes\delta_{\psi^\pm_{k}}(B\times D)<y\right)\\& 	=\mathbb{P}\bigg(\xi'_{\pm}\in\{\mu\in\mathcal{M}_{S\times T}:\sum_{k}\delta_{\sigma^{\mu}_{k}}\otimes\delta_{\mu(\{\sigma^{\mu}_{k}\}\times\cdot)}(A\times C)<x\},\\&\quad\quad\xi'_{\pm}\in\{\mu\in\mathcal{M}_{S\times T}:\sum_{k}\delta_{\sigma^{\mu}_{k}}\otimes\delta_{\mu(\{\sigma^{\mu}_{k}\}\times\cdot)}(B\times D)<y\}\bigg)\\&	=\mathbb{P}\bigg(\xi'_{\pm}((A\times T)\cap\cdot)\in\{\mu\in\mathcal{M}_{S\times T}:\sum_{k}\delta_{\sigma^{\mu}_{k}}\otimes\delta_{\mu(\{\sigma^{\mu}_{k}\}\times\cdot)}(A\times C)<x\},\\&\quad\quad	\xi'_{\pm}(B\times T)\cap\cdot\in\{\mu\in\mathcal{M}_{S\times T}:\sum_{k}\delta_{\sigma^{\mu}_{k}}\otimes\delta_{\mu(\{\sigma^{\mu}_{k}\}\times\cdot)}(B\times D)<y\}\bigg)\\&	=\mathbb{P}\bigg(\xi'_{\pm}((A\times T)\cap\cdot)\in\{\mu\in\mathcal{M}_{S\times T}:\sum_{k}\delta_{\sigma^{\mu}_{k}}\otimes\delta_{\mu(\{\sigma^{\mu}_{k}\}\times\cdot)}(A\times C)<x\}\bigg)\\&	\quad\quad\mathbb{P}\bigg(\xi'_{\pm}(B\times T)\cap\cdot\in\{\mu\in\mathcal{M}_{S\times T}:\sum_{k}\delta_{\sigma^{\mu}_{k}}\otimes\delta_{\mu(\{\sigma^{\mu}_{k}\}\times\cdot)}(B\times D)<y\}\bigg)
	\end{align*}
	\begin{align*}
		& 	=\mathbb{P}\bigg(\sum_{k}\delta_{\sigma^{\pm}_{k}}\otimes\delta_{\psi^\pm_{k}}(A\times C)<x\Big)\mathbb{P}\Big(\sum_{k}\delta_{\sigma^{\pm}_{k}}\otimes\delta_{\psi^\pm_{k}}(B\times D)<y\bigg)\\&			=	\mathbb{P}\Big(\gamma_\pm(A\times C)<x\Big)\mathbb{P}\Big(\gamma_\pm(B\times D)<y\Big),
	\end{align*}
	where $\sigma^{\mu}_{k}$ indicates the $k$-th atoms of the measure $\mu$ and where we used that sets of the form of $A\times C$ and $B\times D$ are the sets generating $\hat{\mathcal{S}}\otimes\hat{\mathcal{B}}_{\mathcal{M}_\mathcal{T}\setminus\{0\}}$. Thus, we obtain that $\gamma_+-\gamma_-$ is an $(\mathcal{M}_T\setminus\{0\})$-marked signed point process with independent $S$-increments. 
	
	By Proposition \ref{pro-marked-signed} we have that there exist two independent Poisson processes $\eta_+$ and $\eta_-$ on $S\times( \mathcal{M}_T\setminus\{0\})$ such that $\gamma_+-\gamma_-\stackrel{a.s.}{=}\eta_+-\eta_-$. By Proposition \ref{pro-marked-signed}  $\eta_+$ and $\eta_-$ are the Jordan decomposition of $\gamma_+-\gamma_-$. Hence,
	\begin{equation*}
		\alpha_++\int\int (\delta_s\otimes\mu)\eta_+(ds\,d\mu)\quad\textnormal{and}\quad \alpha_-+\int\int (\delta_s\otimes\mu)\eta_-(ds\,d\mu)
	\end{equation*}
	are also the Jordan decomposition of $\xi'$, which implies that $\xi'_+$ and $\xi'_-$ are independent. Since $\gamma_+$ depends measurably on  $\xi'_+$ (and $\gamma_-$ on $\xi'_-$ ) we obtain that $\gamma_+$  and $\gamma_-$ are independent. Hence, $\gamma_+\stackrel{a.s.}{=}\eta_+$ and $\gamma_-\stackrel{a.s.}{=}\eta_-$. The rest of the statement follows from Theorem 3.19 in \cite{Kallenberg2} (we remark that the statement of Theorem 3.19 in \cite{Kallenberg2} has a typo in equation (8), it should be $C\in\hat{\mathcal{T}}$ and not $C\in\hat{\mathcal{M}}_T$).
\end{proof}

\begin{proof}[Proof of Lemma \ref{lem-Borel-disjoint-union}]
	It follows from the fact that the disjoint union of Borel spaces is a Borel space (see Section 12.B in \cite{Kechris}).
\end{proof}

\begin{proof}[Proof of Corollary \ref{co-incited-5}]
	It follows from Theorem \ref{thm-ultra-representation} and Lemma \ref{lem-Borel-disjoint-union}.
\end{proof}

\begin{proof}[Proof of Corollary \ref{co-ultra-two-CRSM-2}]
	It follows from Theorem \ref{thm-ultra-representation} and Lemma \ref{lem-Borel-disjoint-union} using the same arguments as in the proof of Corollary \ref{co-two-CRSM-2}.
\end{proof}

\subsection*{Proofs of Sections \ref{Sec-Examples} and  \ref{Sec-applications}}

\begin{proof}[Proof of Proposition \ref{pro-Skellam}]
	Since $\xi$ has independent increments then $\xi_+$ and $\xi_-$ are two independent CRMs without fixed atoms. The mass of each atom is 1, hence by Theorem 3.17 in \cite{Kallenberg2} $\xi_+$ and $\xi_-$ are two independent Poisson processes. The statement for $\eta$ follows by definition.
\end{proof}

\begin{proof}[Proof of Proposition \ref{pro-Bay}]
	The proof generalizes the arguments used in the proof of Theorem 3.1 in \cite{Bro1}. To lighten the notation let $X:=X_1$. Let us first prove the result for $\Theta|X$. Any fixed atom $\theta_{fix,k}\delta_{\psi_{fix,k}}$ in the prior is independent of the other fixed atoms and of the ordinary component. Thus, all of $X$ except $x_{fix,k}:=X(\{\psi_{fix,k}\})$ is independent of $\theta_{fix,k}$. Hence, $\Theta|X$ has a fixed atom at $\psi_{fix,k}$ and by Bayes theorem we obtain $F_{post,fix,k}(d\theta)\varpropto F_{fix,k}(d\theta)h(x_{fix,k}|\theta)$.
	
	Since $G$ is continuous, all the fixed and non-fixed atoms of $\Theta$ are at a.s.~distinct locations. By letting $\Psi_{fix}:=\{\psi_{fix,1},...,\psi_{fix,K_{fix}}\}$ we can define the fixed and ordinary component of $X$ by $X_{fix}(A):=X(A\cap \Psi_{fix})$ and $X_{ord}(A):=X(A\cap (\Psi\setminus\Psi_{fix}))$, respectively.
	
	Let $\{\psi_{new,x,1},...,\psi_{new,x,K_{new,x}}\}$ be all the locations of atoms in $X_{ord}$ of size $x\in\mathbb{Z}\setminus\{0\}$, which is finite and it is a subset of the locations of atoms of $\Theta_{ord}$ by assumption A2. Further, let $\theta_{new,x,k}:=\Theta(\{\psi_{new,x,k}\})$. The values $\{\theta_{new,x,k} \}_{k=1}^{K_{new,x}}$ are generated from a thinned Poisson point process with intensity measure $\nu_{x}(d\theta)=\nu(d\theta)h(x|\theta)$ on $\mathbb{R}\setminus\{0\}$, and this is due to the $h(x|\theta)$-thinning of the Poisson point process $\sum_{k=1}^{K_{ord}}\delta_{\theta_{ord,k}}$ which has intensity measure $\nu$ on $\mathbb{R}\setminus\{0\}$. Further, since $\int_{\mathbb{R}\setminus\{0\}}h(x|\theta)\nu(d\theta)<\infty$ by assumption A2, we have that $\mathcal{L}(\theta_{post,new,x,k})\varpropto\nu(d\theta)h(x|\theta)$.
	
	When the likelihood draw returns a zero, the atoms in $\Theta_{ord}$ are not observed in $X_{ord}$. Thus, by proceeding as in the previous paragraph with $x=0$ we obtain the last statement.
	
	Considering $\Theta|X_{1}$ as the new prior we obtain the formulation for the posterior $\Theta|X_{1},X_{2}$ by induction and by observing that the assumptions are still satisfied by $\Theta|X_{1}$. Then, by induction we conclude the proof.
\end{proof}

\begin{proof}[Proof of Proposition \ref{pro-Bay-2}]
	The arguments are essentially the ones of the proof of Proposition \ref{pro-Bay}. We only write the ones for the second statement explicitly because they require more attention. Let $A\in\mathcal{B}(\mathbb{R}\setminus\{0\})$ and let $\{\psi_{new,A,1},...,\psi_{new,A,K_{new,A}}\}$ be all the locations of atoms in $X_{ord}$ of size $x\in A$, which is finite and it is a subset of the locations of atoms of $\Theta_{ord}$ by assumption A2'. Let $\theta_{new,A,k}:=\Theta(\{\psi_{new,A,k}\})$. The values $\{\theta_{new,A,k} \}_{k=1}^{K_{new,A}}$ are generated from a thinned Poisson point process with intensity $\nu_{A}(d\theta)=\nu(d\theta)H_{ord}(A|\theta)$ on $\mathbb{R}\setminus\{0\}$, and this is due to the $H_{ord}(A|\theta)$-thinning of the Poisson point process $\sum_{k=1}^{K_{ord}}\delta_{\theta_{ord,k}}$ which has intensity measure $\nu$ on $\mathbb{R}\setminus\{0\}$. Further, since $\int_{\mathbb{R}\setminus\{0\}}H_{ord}(A|\theta)\nu(d\theta)<\infty$ by assumption A2', we have that $\mathcal{L}(\theta_{post,new,A,k})\varpropto\nu(d\theta)H_{ord}(A|\theta)$.
\end{proof}

\begin{proof}[Proof of Proposition \ref{pro-exchange}]
	It is possible to see that $Z$ is the difference of two independent random measures $Z_+$ and $Z_-$, where $Z_+$ is given by the hierarchical model in equation (17) in \cite{Caron-and-Fox} where the generative CRM is $W_+$  and the multigraph is $D_+$, and similarly $Z_-$ is given by $W_-$ and $D_-$. Since by Proposition 1 in \cite{Caron-and-Fox} we have that $Z_+$ and $Z_-$ are jointly exchangeable, by independence we obtain the result.
\end{proof}

\begin{proof}[Proof of Proposition \ref{pro-sparse}]
	Recall the definition of $Z_+$ and $Z_-$ from the proof of Proposition \ref{pro-exchange}. It is possible to notice that $Z_+$ and $Z_-$ are the Jordan decomposition of $Z$. Indeed, $Z_+$ and $Z_-$ have almost surely distinct atoms. Moreover, let 
	\begin{align*}
		N_{\alpha,+}&:=\textnormal{card}\{\theta_i\in[0,\alpha]:Z_+(\{\theta_i\}\times[0,\alpha])\neq 0\}\\&=\textnormal{card}\{\theta_i\in[0,\alpha]:Z_+(\{\theta_i\}\times[0,\alpha])> 0\},\\
		N_{\alpha,+}^{(e)}&:=Z_+(\{(x,y)\in[0,\infty)^2: 0\leq x\leq y\leq\alpha\}),\\
		N_{\alpha,-}&:=\textnormal{card}\{\theta_i\in[0,\alpha]:Z_-(\{\theta_i\}\times[0,\alpha])\neq 0\}\\&=\textnormal{card}\{\theta_i\in[0,\alpha]:Z_+(\{\theta_i\}\times[0,\alpha])< 0\},\\
		N_{\alpha,-}^{(e)}&:=Z_-(\{(x,y)\in[0,\infty)^2: 0\leq x\leq y\leq\alpha\}).
	\end{align*}
	Observe that $N_\alpha\stackrel{a.s.}{=}N_{\alpha,+}+N_{\alpha,-}$ and $N_\alpha^{(e)}\stackrel{a.s.}{=}N_{\alpha,+}^{(e)}+N_{\alpha,-}^{(e)}$. The latter follows from the fact that $Z_+$ and $Z_-$ are the Jordan decomposition of $Z$. The former follows from the fact that when $w_i(\omega)> 0$, for some $\omega\in\Omega$, then $\theta_i$ can only have connections with the $\theta$s whose weights $w$s are positive (and similarly when $w_i(\omega)< 0$). 
	
	Since $N_{\alpha,+}$ and $N_{\alpha,+}^{(e)}$ are the number of observed nodes and the number of the edges of $Z_+$ restricted on the square $[0,\alpha]^2$ (and similarly for $Z_-$), by Theorem 2 in \cite{Caron-and-Fox} we obtain the following. When $\int_{\mathbb{R}\setminus\{0\}}\rho(dw)<\infty$, namely when $W_+$ and $W_-$ have finite activity (\textit{i.e.}~$\int_{0}^{\infty}\rho(dw)<\infty$ and $\int_{-\infty}^{0}\rho(dw)<\infty$), we have $N_{\alpha,+}^{(e)}=\Theta(N_{\alpha,+}^2)$ and $N_{\alpha,-}^{(e)}=\Theta(N_{\alpha,-}^2)$ almost surely as $\alpha\to\infty$. Hence, almost surely as $\alpha\to\infty$
	\begin{equation*}
		N_{\alpha}^{(e)}=N_{\alpha,+}^{(e)}+N_{\alpha,-}^{(e)}=\Theta(N_{\alpha,+}^2)+\Theta(N_{\alpha,-}^2)=\Theta((N_{\alpha,+}+N_{\alpha,-})^2)=\Theta(N_{\alpha}^2),
	\end{equation*}
	When $\int_{\mathbb{R}\setminus\{0\}}\rho(dw)=\infty$, namely when at least one of $W_+$ and $W_-$ has infinite activity we have that $N_{\alpha,+}^{(e)}=o(N_{\alpha,+}^2)$ or $N_{\alpha,-}^{(e)}=o(N_{\alpha,-}^2)$ or both almost surely as $\alpha\to\infty$, and therefore $N_{\alpha}^{(e)}=o(N_{\alpha}^2)$ almost surely as $\alpha\to\infty$.
\end{proof}
	
	\bibliographystyle{chicago}
	\bibliography{bib-Random-measures}
	
	
\end{document}